\documentclass{article}
\usepackage{cite}
\usepackage{amsmath,amssymb,amsthm}
\usepackage{titlesec,hyperref}
\usepackage{color}
\usepackage{fancyhdr}
\usepackage[margin=2.5cm]{geometry}
\usepackage{enumerate}
\usepackage[integrals]{wasysym}
\usepackage{bm}
\usepackage{cases}
\usepackage{graphicx}
\usepackage{caption}
\usepackage{float}
\usepackage{subfigure}
\usepackage[sort&compress,numbers]{natbib}
\usepackage{lipsum}

\newtheorem{theo}{Theorem}[section]
\newtheorem{lemm}[theo]{Lemma}

\newtheorem{prop}[theo]{Proposition}
\newtheorem{rema}[theo]{Remark}
\numberwithin{equation}{section}

\def\ud{{\rm d}}

\allowdisplaybreaks

\begin{document}

\title{The continuous dependence and non-uniform dependence of the rotation Camassa-Holm equation in Besov spaces}
\author{
    Yingying Guo$^{*}$~and~Xi Tu \\
    Department of Mathematics, Foshan University,\\
    Foshan, 528000, China
}
\date{}
\maketitle
\newcommand\blfootnote[1]{%
	\begingroup
	\renewcommand\thefootnote{}\footnote{#1}%
	\addtocounter{footnote}{-1}%
	\endgroup
}

\noindent ABSTRACT. In this paper, we first establish the local well-posedness and continuous dependence for the rotation Camassa-Holm equation modelling the equatorial water waves with the weak Coriolis effect in nonhomogeneous Besov spaces $B^s_{p,r}$ with $s>1+1/p$ or $s=1+1/p,\ p\in[1,+\infty),\ r=1$ by a new way: the compactness argument and Lagrangian coordinate transformation, which removes the index constraint $s>3/2$ and improves our previous work \cite{guoy1}. Then, we prove the solution is not uniformly continuous dependence on the initial data in both supercritical and critical Besov spaces.

\blfootnote{\noindent \textit{Mathematics Subject Classification.} Primary: 35Q35; Secondary: 35B30.}
\blfootnote{\noindent \textit{Keywords.} The rotation Camassa-Holm equation, Besov spaces, Continuous dependence, Non-uniform dependence.}
\blfootnote{*Corresponding Author: Yingying Guo.}


\section{Introduction}
\par
In this paper, we consider the Cauchy problem for the following rotation-Camassa-Holm (R-CH) equation in the equatorial water waves with the weak Coriolis effect \cite{cgl}
\begin{numcases}{}
u_t-\beta\mu u_{xxt}+cu_x+3\alpha\varepsilon uu_{x}-\beta_0\mu u_{xxx}+\omega_1\varepsilon^2u^2u_x+\omega_2\varepsilon^3u^3u_x=\alpha\beta\varepsilon\mu(2u_xu_{xx}+uu_{xxx}),\label{eq0}\\
u(0,x)=u_0(x),\label{eq}
\end{numcases}
where $\varepsilon$ is the amplitude parameter, $\mu$ is the shallowness, $\Omega$ characterizes the constant rotational speed of the Earth, the other coefficients of \eqref{eq0} are defined as
\begin{align*}
&c=\sqrt{1+\Omega^2}-\Omega,\ \alpha=\frac{c^2}{1+c^2},\ \beta_0=\frac{c(c^4+6c^2-1)}{6(c^2+1)^2},\ \beta=\frac{3c^4+8c^2-1}{6(c^2+1)^2},\\
&\omega_1=-\frac{3c(c^2-1)(c^2-2)}{2(c^2+1)^3},\ \omega_2=\frac{(c^2-1)^2(c^2-2)(8c^2-1)}{2(c^2+1)^5}.
\end{align*}
Note that $(1-\partial_x^2)^{-1}f={\rm{p}}\ast f$ for any $f\in L^2$, where $\ast$ denote the convolution and ${\rm{p}}(x)=\frac{1}{2}e^{-|x|}$.
Applying the scaling 
$$x\rightarrow x-c_0t,\quad u\rightarrow u-\gamma,\quad t\rightarrow t$$
where $c_0=\frac{\beta_0}{\beta}-\gamma$ and $\gamma$ is the real root of $c-\frac{\beta_0}{\beta}-2\gamma+\frac{\omega_1}{\alpha^2}\gamma^2-\frac{\omega_2}{\alpha^3}\gamma^3=0$, we can rewrite \eqref{eq0} in the weak form
\begin{equation}\label{scaling}
u_t+uu_x=-\partial_{x}p\ast \left(\frac{1}{2}u_x^2+c_1u^2+c_2u^3+c_3u^4\right):=G(u),
\end{equation}
where 
\begin{align*}
c_1=1+\frac{3\gamma^2\omega_2}{2\alpha^3}-\frac{\omega_1\gamma}{\alpha^2},\ c_2=\frac{\omega_1}{3\alpha^2}-\frac{\omega_2\gamma}{\alpha^3},\ c_3=\frac{\omega_2}{4\alpha^3}.
\end{align*}

The study of nonlinear equatorial geophysical waves is of great current interest. The Earth's rotation affects the atmosphere-ocean flow near the Equator in such a way that waves propagate practically along the Equator \cite{cj1}. Furthermore, field data \cite{cj2} suggests that in a 300 km wide strip centered on the Equator, one can use the $f$-plane approximation to justify the relevance of two-dimensional models like \eqref{eq0}. Following the derivations of the Camassa-Holm equation \cite{j} and the Constantin-Lannes equation \cite{cl}, Chen et al. \cite{cgl} derived the R-CH equation \eqref{eq0} in the equatorial region with the influence of the gravity effect and the Coriolis effect from the f-plane governing equation. In addition, the model is analogous to the rotation-Green-Naghdi (R-GN) equations with the weak Earth's rotation effect and Chen et al.\cite{cgl} justified that the R-GN equations tend to associated solution of the R-CH equations in the Camassa-Holm regime $\mu\ll1,\ \varepsilon=O(\sqrt{\mu})$.

For $\Omega=0$, then the coefficients $c=1,\ \alpha=\frac 1 2,\ \beta_0=\frac 1 4,\ \beta=\frac{5}{12},\ \omega_1=\omega_2=0$. Eq. \eqref{eq0} becomes
\begin{equation}\label{eq1}
u_t - \frac{5}{12}\mu u_{xxt}+u_x+\frac{3}{2}\varepsilon uu_{x}-\frac{1}{4}\mu u_{xxx}=\frac{5}{24}\varepsilon\mu(2u_xu_{xx}+uu_{xxx}).
\end{equation}
Applying the transformation
$$u(t,x)=\alpha\varepsilon u(\sqrt{\beta\mu}t,\sqrt{\beta\mu}x)$$
and then using the Galilean transformation
$$u(t,x)\rightarrow u(t,x-\frac{3}{4}t)+\frac{1}{4},$$
Eq. \eqref{eq1} becomes the following classical Camassa-Holm (CH) equation describing the motion of waves at free surface of shallow water under the influence of gravity \cite{ch,cl}
\begin{equation}\label{eq2}
u_t - u_{xxt}+3uu_{x}=2u_{x}u_{xx}+uu_{xxx}.
\end{equation}
The CH equation is completely integrable \cite{cgi} and has a bi-Hamiltonian structure \cite{ff}. It also has the solitary waves and peak solitons \cite{cs}. It is worth mentioning that the peakons show the characteristic for the traveling waves of greatest height and arise as solutions to the free-boundary problem for the incompressible Euler equations over a flat bed, see \cite{c5,t}. The local well-posedness, global strong solutions, blow-up strong solutions of the CH equations were studied in \cite{c2,ce2,ce3,d1,lio}. The global weak solutions, global conservative solutions and dissipative solutions also have been investigated in \cite{bc1,bc2,bcz,cmo,xz1}. For the continuity of the solutions map of the CH equations with respect to the initial data, it was only proved in the spaces $C([0,T];B^{s'}_{p,r})$ for any $s'<s$ with $s>\max\{3/2,1+1/p\}$ by many authors. Recently, Li and Yin \cite{liy} proved that the index of the continuous dependence of the solutions for the Camassa-Holm type equations in $B^{s}_{p,r} \big(s>\max\{3/2,1+1/p\}\big)$ can up to $s$, which improved many authors' results, especially the Danchin's results in \cite{d1,d2}. More recently, Guo et al. \cite{glmy} obtained the local ill-posedness for a class of shallow water wave equations (such as, the CH, DP, Novikov equations and etc.) in critical Sobolev space $H^{\frac 3 2}$ and even in Besov space $B^{1+1/p}_{p,r}$ with $p\in[1,+\infty],\ r\in(1,+\infty]$. However, whether the CH equation is local well-posed or not in critial Besov spaces $B^{1+1/p}_{p,1},p\in(2,+\infty]$ is still an open problem. The main difficult is that the CH equation induce a loss of one order dervative in the stability estimates. To overcome this difficult, we adopt the compactness argument and Lagrangian coordinate transformation in our upcoming article \cite{yyg} rather than the usual techniques used in \cite{liy,guoy1} to obtain the local well-posedness and continuous dependence for the Cauchy problem of CH equation in critial Besov spaces $B^{1+1/p}_{p,1}$ with $p\in[1,+\infty)$. This implies $B^{1+1/p}_{p,1}$ is the critical Besov space and the index $\frac 3 2$ is not necessary for the Camassa-Holm type equations. The locally well-posed or not in $B^{1}_{\infty,1}$ is more difficult for $B^{0}_{\infty,1}$ is not a Banach algebra and is what we're going to consider next. Further, the non-uniform continuity of the CH equation has been investigated in many papers, see \cite{hmp,hk,hkm,lyz,lwyz}. 

We know that the nonlinear terms in the CH equation are all quadratic, but due to the influence of Earth's deflection force caused by the rotation of the Earth, there are three or even four nonlinear terms in the R-CH model, which have an important influence on the fluid movement, especially the phenomenon of wave splitting. So this model has attracted some attention and got some results. Zhu, Liu and Ming \cite{zlm} studied the wave-breaking phenomena and persistence properties for Eq. \eqref{scaling}. Tu, Liu and Mu \cite{tlm} investigated the existence and uniqueness of the global conservative weak solutions to Eq. \eqref{scaling}.
Moreover, the non-uniform dependence about initial data on the circle for Eq. \eqref{scaling} in Sobolev spaces was studied in \cite{z}. 

However, the local well-posedness of the Cauchy problem for Eq. \eqref{scaling} in critial Besov spaces $B^{1+1/p}_{p,1}(\mathbb{R}),\ p\in(2,+\infty]$ and the non-uniform dependence on initial data in $B_{p,r}^{s},\ s>\max\{3/2,1+1/p\}$ or $s=1+1/p,\ p\in[1,2],\ r=1$ have not been investigated yet. In the paper, following the idea of our upcoming article \cite{yyg} and \cite{lyz,lwyz}, we aim to study the local well-posedness and the non-uniform dependence on initial data for Eq. \eqref{scaling} in both supercritical and critical Besov spaces.\\

our main results are stated as follows.
\begin{theo}\label{th}
Let $s\in\mathbb{R},\ 1\leq p,r \leq\infty$ and let $(s,p,r)$ satisfy the condition
\begin{equation}\label{condition}
s>1+\frac 1 p\quad\text{or}\quad s=1+\frac{1}{p},\ 1\leq p<+\infty,\ r=1.
\end{equation}
Assume $u_{0}\in B^s_{p,r}(\mathbb{R}).$ Then, there exists a $T>0$ such that Eq. \eqref{scaling} has a unique solution $u$ in $E^s_{p,r}(T)$ with the initial data $u_0$ and the map $u_0\mapsto u$ is continuous from any bounded subset
of $B^s_{p,r}$ into $E^s_{p,r}(T)$ . Moreover, for all $t\in[0,T]$, we have 
\begin{align}
\|u(t)\|_{B^s_{p,r}}\leq C\|u_0\|_{B^s_{p,r}}.\label{unformly}
\end{align}
\end{theo}
\begin{rema}
The index constraint $\frac{3}{2}$ is removed in Theorem \ref{th}, which implies that $\frac{3}{2}$ is not necessary for the rotation Camassa-Holm equation. This can be achieved by the compactness argument and Lagrangian coordinate transformation, which improves our previous work \cite{guoy1}, see the proof of Theorem \ref{th} for more details in Section 3.
\end{rema}
\begin{theo}\label{uniform}
Let $s\in\mathbb{R},\ 1\leq p,\ r\leq\infty$ and let $(s,p,r)$ satisfy the condition
\begin{equation}\label{cond}
s>\max\big\{\frac 3 2,1+\frac 1 p\big\},\ \ 1\leq p\leq\infty,\ 1\leq r <\infty.
\end{equation}
Then the solution map of problem \eqref{scaling}-\eqref{eq} is not uniformly continuous from any bounded subset in $B^{s}_{p,r}$ into $C([0,T];B^{s}_{p,r})$. More precisely, there exists two sequences of solutions $u^n$ and ${\rm w}^n$ with the initial data $u^n_0={\rm w}^n_0+{\rm v}_0^n$ and ${\rm w}^n_0$ such that 
\begin{align*}
\|{\rm w}^n_0\|_{B^{s}_{p,r}}\lesssim 1\qquad and \qquad \lim\limits_{n\rightarrow\infty}\|{\rm v}^n_0\|_{B^{s}_{p,r}}=0,
\end{align*}
but
\begin{align*}
\liminf\limits_{n\rightarrow\infty}\Big\|u^n-{\rm w}^n\Big\|_{B^{s}_{p,r}}\gtrsim t,	
\ \forall\ t\in[0,T_0],
\end{align*}
with small time $T_0\leq T$.
\end{theo}
\begin{rema}
For the non-uniform denpendence of the solutions to \eqref{scaling} in supercritical Besov spaces, we inevitably use the Bony decomposition theory to estimate the $B^{s}_{p,r}$- norm of the initial data we constructed. So we still keep the constraint $\frac 3 2 $ in Theorem \ref{uniform}. The key argument is to construct the initial data.
\end{rema}
\begin{theo}\label{uniform1}
Let $s\in\mathbb{R},\ 1\leq p,\ r\leq\infty$ and let $(s,p,r)$ satisfy the condition
\begin{equation}\label{cond1}
s=1+\frac{1}{p},\ 1\leq p\leq2,\ r=1.
\end{equation}
Then the solution map of problem \eqref{scaling}-\eqref{eq} is not uniformly continuous from any bounded subset in $B^{1+\frac{1}{p}}_{p,1}$ into $C([0,T];B^{1+\frac{1}{p}}_{p,1})$. To be more exact, there exists two sequences of solutions $u^n$ and ${\rm w}^n$ with the initial data $u^n_0={\rm w}^n_0+{\rm v}_0^n$ and ${\rm w}^n_0$ such that 
\begin{align*}
\|{\rm w}^n_0\|_{B^{1+\frac{1}{p}}_{p,1}}\lesssim 1\qquad and \qquad \lim\limits_{n\rightarrow\infty}\|{\rm v}^n_0\|_{B^{1+\frac{1}{p}}_{p,1}}=0,
\end{align*}	
but
\begin{align*}
\liminf\limits_{n\rightarrow\infty}\Big\|u^n-{\rm w}^n\Big\|_{B^{1+\frac{1}{p}}_{p,1}}\gtrsim t,\ \forall\ t\in[0,T_0],
\end{align*}
with small time $T_0\leq T$.
\end{theo}
\begin{rema}
For the critial case, the transport equation theory prevents us to obtaining the estimate of the solutions in $B^{\frac 1 p}_{p,1}$. To overcome the problem, we in turn estimate the $L^{\infty}$ norm instead of the $B^{\frac 1 p}_{p,1}$ norm of the solutions, which is different from the supercritial case, see the proof of Theorem \ref{uniform1} in Section 4.
\end{rema}

Our paper unfolds as follows. In the second section, we introduce some preliminaries which will be used in this sequel. In the third section, we establish the local well-posedness and continuous dependness of \eqref{scaling} in $B^s_{p,r}$ with $s>1+\frac{1}{p}$ or $s=1+\frac{1}{p},\ p\in[1,+\infty),\ r=1$. In the last section, we give the non-uniform dependence on initial data for \eqref{scaling} in $B_{p,r}^{s},\ s>\max\{\frac{3}{2},1+\frac{1}{p}\},\ p\in[1,+\infty],\ r\in[1,+\infty)$ or $s=1+\frac{1}{p},\ p\in[1,2],\ r=1$.

\section{Preliminaries}
\par
In this section, we first introduce some properties of the Littlewood-Paley theory in \cite{book}.

Let $\chi: \mathbb{R}\rightarrow[0,1]$ be a radical, smooth, and even function which is suppported in $\mathcal{B}=\{\xi:|\xi|\leq\frac 4 3\}$.
Let $\varphi:\mathbb{R}\rightarrow[0,1]$ be a radical, smooth, function which is suppported in $\mathcal{C}=\{\xi:\frac 3 4\leq|\xi|\leq\frac 8 3\}$.

Denote $\mathcal{F}$ and $\mathcal{F}^{-1}$ by the Fourier transform and the Fourier inverse transform respectively as follows:
\begin{align*}
&\mathcal{F}u(\xi)=\hat{u}(\xi)=\int_{\mathbb{R}}e^{-ix\xi}u(x){\ud}x,\\
&u(x)=\big(\mathcal{F}^{-1}\hat{u}\big)(x)=\frac{1}{2\pi}\int_{\mathbb{R}}e^{ix\xi}\hat{u}(\xi){\ud}\xi.	
\end{align*}

For any $u\in\mathcal{S}'(\mathbb{R}^d)$ and all $j\in\mathbb{Z}$, define
$\Delta_j u=0$ for $j\leq -2$; $\Delta_{-1} u=\mathcal{F}^{-1}(\chi\mathcal{F}u)$; $\Delta_j u=\mathcal{F}^{-1}(\varphi(2^{-j}\cdot)\mathcal{F}u)$ for $j\geq 0$; and $S_j u=\sum_{j'<j}\Delta_{j'}u$.

Let $s\in\mathbb{R},\ 1\leq p,r\leq\infty.$ We define the nonhomogeneous Besov space $B^s_{p,r}(\mathbb{R}^d)$
$$  B^s_{p,r}=B^s_{p,r}(\mathbb{R}^d)=\Big\{u\in S'(\mathbb{R}^d):\|u\|_{B^s_{p,r}}=\big\|(2^{js}\|\Delta_j u\|_{L^p})_j \big\|_{l^r(\mathbb{Z})}<\infty\Big\}.$$

The corresponding nonhomogeneous Sobolev space $H^{s}(\mathbb{R}^d)$ is
$$H^{s}=H^{s}(\mathbb{R}^d)=\Big\{u\in S'(\mathbb{R}^d):\ u\in L^2_{loc}(\mathbb{R}^d),\ \|u\|^2_{H^s}=\int_{\mathbb{R}^d}(1+|\xi|^2)^s|\mathcal{F}u(\xi)|^2d\xi<\infty\Big\}. $$

We introduce a function space, which will be used in the following.
\begin{equation*}
    E^s_{p,r}(T)\triangleq \left\{\begin{array}{ll}
    C([0,T];B^s_{p,r})\cap C^1([0,T];B^{s-1}_{p,r}), & \text{if}\ r<\infty,  \\
    C_w([0,T];B^s_{p,\infty})\cap C^{0,1}([0,T];B^{s-1}_{p,\infty}), & \text{if}\ r=\infty.
    \end{array}\right.
\end{equation*}

Then, we recall some properties about the Besov spaces.
\begin{prop}\label{Besov}\cite{book}
Let $s\in\mathbb{R},\ 1\leq p,p_1,p_2,r,r_1,r_2\leq\infty.$  \\
{\rm(1)} $B^s_{p,r}$ is a Banach space, and is continuously embedded in $\mathcal{S}'$. \\
{\rm(2)} If $r<\infty$, then $\lim\limits_{j\rightarrow\infty}\|S_j u-u\|_{B^s_{p,r}}=0$. If $p,r<\infty$, then $C_0^{\infty}$ is dense in $B^s_{p,r}$. \\
{\rm(3)} If $p_1\leq p_2$ and $r_1\leq r_2$, then $ B^s_{p_1,r_1}\hookrightarrow B^{s-(\frac d {p_1}-\frac d {p_2})}_{p_2,r_2}. $
If $s_1<s_2$, then the embedding $B^{s_2}_{p,r_2}\hookrightarrow B^{s_1}_{p,r_1}$ is locally compact. \\
{\rm(4)} $B^s_{p,r}\hookrightarrow L^{\infty} \Leftrightarrow s>\frac d p\ \text{or}\ s=\frac d p,\ r=1$. \\
{\rm(5)} Fatou property: if $(u_n)_{n\in\mathbb{N}}$ is a bounded sequence in $B^s_{p,r}$, then an element $u\in B^s_{p,r}$ and a subsequence $(u_{n_k})_{k\in\mathbb{N}}$ exist such that
$$ \lim_{k\rightarrow\infty}u_{n_k}=u\ \text{in}\ \mathcal{S}'\quad \text{and}\quad \|u\|_{B^s_{p,r}}\leq C\liminf_{k\rightarrow\infty}\|u_{n_k}\|_{B^s_{p,r}}. $$
{\rm(6)} Let $m\in\mathbb{R}$ and $f$ be a $S^m$-mutiplier $($i.e. f is a smooth function and satisfies that $\forall\ \alpha\in\mathbb{N}^d$,
$\exists\ C=C(\alpha)$ such that $|\partial^{\alpha}f(\xi)|\leq C(1+|\xi|)^{m-|\alpha|},\ \forall\ \xi\in\mathbb{R}^d)$.
Then the operator $f(D)=\mathcal{F}^{-1}(f\mathcal{F})$ is continuous from $B^s_{p,r}$ to $B^{s-m}_{p,r}$.
\end{prop}

We next give some crucial interpolation inequalities.
\begin{prop}\label{prop}\cite{book}
{\rm(1)} If $s_1<s_2$, $\lambda\in (0,1)$, and $(p,r)\in[1,\infty]^2$, then we have
\begin{align*}
	&\|u\|_{B^{\lambda s_1+(1-\lambda)s_2}_{p,r}}\leq \|u\|_{B^{s_1}_{p,r}}^{\lambda}\|u\|_{B^{s_2}_{p,r}}^{1-\lambda},\\
	&\|u\|_{B^{\lambda s_1+(1-\lambda)s_2}_{p,1}}\leq\frac{C}{s_2-s_1}\Big(\frac{1}{\lambda}+\frac{1}{1-\lambda}\Big) \|u\|_{B^{s_1}_{p,\infty}}^{\lambda}\|u\|_{B^{s_2}_{p,\infty}}^{1-\lambda}.
\end{align*}
{\rm(2)} If $s\in\mathbb{R},\ 1\leq p\leq\infty,\ \varepsilon>0$, a constant $C=C(\varepsilon)$ exists such that
$$ \|u\|_{B^s_{p,1}}\leq C\|u\|_{B^s_{p,\infty}}\ln\Big(e+\frac {\|u\|_{B^{s+\varepsilon}_{p,\infty}}}{\|u\|_{B^s_{p,\infty}}}\Big). $$
\end{prop}

The 1-D Moser-type estimates are provided as follows.
\begin{lemm}\label{product}\cite{book} The following estimates hold:\\
{\rm(1)} For any $s>0$ and any $p,\ r$ in $[1,\infty]$, the space $L^{\infty}\cap B^s_{p,r}$ is an algebra, and a constant $C=C(s)$ exists such that
\begin{align*}
\|uv\|_{B^s_{p,r}}\leq C(\|u\|_{L^{\infty}}\|v\|_{B^s_{p,r}}+\|u\|_{B^s_{p,r}}\|v\|_{L^{\infty}}).
\end{align*}
{\rm(2)} If $s>\max\{\frac{3}{2},1+\frac{1}{p}\}$ and $1\leq p,\ r\leq\infty$. Then the following inequatily holds
\begin{align*}
	\|uv\|_{B^{s-2}_{p,r}}\leq C\|u\|_{B^{s-2}_{p,r}}\|v\|_{B^{s-1}_{p,r}}.	
\end{align*}
{\rm(3)} If $1\leq p,r\leq \infty,\ s_1\leq s_2,\ s_2>\frac{1}{p} (s_2 \geq \frac{1}{p}\ \text{if}\ r=1)$ and $s_1+s_2>\max(0, \frac{2}{p}-1)$, there exists $C=C(s_1,s_2,p,r)$ such that
$$ \|uv\|_{B^{s_1}_{p,r}}\leq C\|u\|_{B^{s_1}_{p,r}}\|v\|_{B^{s_2}_{p,r}}. $$
\end{lemm}

Here is the useful Gronwall lemma.
\begin{lemm}\label{gwl}\cite{book}
Let $q(t),\ a(t)\in C^1([0,T]),\ q(t),\ a(t)>0$. Let $b(t)$ is a continuous function on $[0,T]$. Suppose that, for all $t\in [0,T]$,
$$\frac{1}{2}\frac{{\ud}}{{\ud}t}q^2(t)\leq a(t)q(t)+b(t)q^2(t).$$
Then for any time $t$ in $[0,T]$, we have
$$q(t)\leq q(0)\exp\int_0^t b(\tau){\ud}\tau+\int_0^t a(\tau)\exp\big(\int_{\tau}^tb(t'){\ud}t'\big){\ud}\tau.$$
\end{lemm}


In the paper, we also need some estimates for the following 1-D transport equation:
\begin{equation}\label{transport}
\left\{\begin{array}{l}
\partial_{t}f+v\partial_{x}f=g, \\
f(0,x)=f_0(x).
\end{array}\right.
\end{equation}

\begin{lemm}\cite{book}\label{existence}
Let $1\leq p\leq\infty,\ 1\leq r\leq\infty$ and $\theta> -\min(\frac 1 {p}, \frac 1 {p'}).$ Suppose $f_0\in B^{\theta}_{p,r},$ $g\in L^1(0,T;B^{\theta}_{p,r}),$ and $v\in L^\rho(0,T;B^{-M}_{\infty,\infty})$ for some $\rho>1$ and $M>0,$ and
\begin{align*}
\begin{array}{ll}
\partial_{x}v\in L^1(0,T;B^{\frac 1 {p}}_{p,\infty}\cap L^{\infty}), &\ \text{if}\ \theta<1+\frac 1 {p}, \\
\partial_{x}v\in L^1(0,T;B^{\theta}_{p,r}),\ &\text{if}\ \theta=1+\frac{1}{p},\ r>1, \\
\partial_{x}v\in L^1(0,T;B^{\theta-1}_{p,r}), &\ \text{if}\ \theta>1+\frac{1}{p}\ (or\ \theta=1+\frac 1 {p},\ r=1).
\end{array}	
\end{align*}
Then the problem \eqref{transport} has a unique solution $f$ in 
\begin{itemize}
\item [-] the space $C([0,T];B^{\theta}_{p,r}),$ if $r<\infty,$
\item [-] the space $\big(\bigcap_{{\theta}'<\theta}C([0,T];B^{{\theta}'}_{p,\infty})\big)\bigcap C_w([0,T];B^{\theta}_{p,\infty}),$ if $r=\infty.$
\end{itemize}
\end{lemm}

\begin{lemm}\cite{book,liy}\label{priori estimate}
Let $1\leq p,\ r\leq\infty$ and $\theta>-\min(\frac{1}{p},\frac{1}{p'}).$
There exists a constant $C$ such that for all solutions $f\in L^{\infty}(0,T;B^{\theta}_{p,r})$ of \eqref{transport} with initial data $f_0$ in $B^{\theta}_{p,r}$ and $g$ in $L^1(0,T;B^{\theta}_{p,r}),$ we have, for a.e. $t\in[0,T],$
$$ \|f(t)\|_{B^{\theta}_{p,r}}\leq \|f_0\|_{B^{\theta}_{p,r}}+\int_0^t\|g(t')\|_{B^{\theta}_{p,r}}{\ud}t'+\int_0^t V'(t')\|f(t')\|_{B^{\theta}_{p,r}}{\ud}t' $$
or
$$ \|f(t)\|_{B^{\theta}_{p,r}}\leq e^{CV(t)}\Big(\|f_0\|_{B^{\theta}_{p,r}}+\int_0^t e^{-CV(t')}\|g(t')\|_{B^{\theta}_{p,r}}{\ud}t'\Big) $$
with
\begin{equation*}
V'(t)=\left\{\begin{array}{ll}
\|\partial_{x}v(t)\|_{B^{\frac 1 p}_{p,\infty}\cap L^{\infty}},\ &\text{if}\ \theta<1+\frac{1}{p}, \\
\|\partial_{x}v(t)\|_{B^{\theta}_{p,r}},\ &\text{if}\ \theta=1+\frac{1}{p},\ r>1, \\
\|\partial_{x}v(t)\|_{B^{\theta-1}_{p,r}},\ &\text{if}\ \theta>1+\frac{1}{p}\ (\text{or}\ \theta=1+\frac{1}{p},\ r=1).
\end{array}\right.
\end{equation*}
If $\theta>0$, then there exists a constant $C=C(p,r,\theta)$ such that the following statement holds
\begin{align*}
\|f(t)\|_{B^{\theta}_{p,r}}\leq \|f_0\|_{B^{\theta}_{p,r}}+\int_0^t\|g(\tau)\|_{B^{\theta}_{p,r}}{\ud}\tau+C\int_0^t \Big(\|f(\tau)\|_{B^{\theta}_{p,r}}\|\partial_{x}v(\tau)\|_{L^{\infty}}+\|\partial_{x}v(\tau)\|_{B^{\theta-1}_{p,r}}\|\partial_{x}f(\tau)\|_{L^{\infty}}\Big){\ud}\tau. 	
\end{align*}
In particular, if $f=av+b,~a,~b\in\mathbb{R},$ then for all $\theta>0,$ $V'(t)=\|\partial_{x}v(t)\|_{L^{\infty}}.$
\end{lemm}

\section{The local well-posedness and continuous dependence in both supercritical and critical Besov spaces}
\par
In this section, we present the local well-posedness and continuous dependence for the Cauchy problem of \eqref{scaling} in Besov spaces $B^{s}_{p,r},\ s>1+\frac{1}{p}$ or $s=1+\frac{1}{p},\ 1\leq p<+\infty,\ r=1$ by the compactness theory and Lagrangian coordinate transformation, which is a new way and is different from our previous proof in \cite{guoy1}.

\begin{proof}[{\rm\textbf{ The proof of Theorem \ref{th}:}}]
Here we just consider the critical case i.e. $s=1+\frac{1}{p},\ 1\leq p<\infty,\ r=1$ in the following, the other case is similar and more easily. Let's divide  four steps to prove it.\\

\textbf{Step 1. We will structure  a family of approximate solution sequences by iterative scheme.}
	
Assuming that $u^0=0$, we define by induction a sequence $\{u^n\}_{n\in\mathbb{N}}$ of smooth functions by solving the following  linear transport equation:
\begin{equation}\label{aa}
\partial_tu^{n+1}+u^n\partial_xu^{n+1}=-\partial_{x}{\rm{p}}\ast\Big\{c_1(u^n)^2+\frac{1}{2}(u_x^n)^2+c_2(u^n)^3+c_2(u^n)^4\Big\}=G(u^{n}).
\end{equation}
Assume that $\{u^n\}_{n\in\mathbb{N}}$ belongs to $C\big([0,T];B^{1+\frac{1}{p}}_{p,1}\big)$ for all $T>0$. We know from Lemma \ref{product} (1) that $B^{\frac{1}{p}}_{p,1},\ B^{1+\frac{1}{p}}_{p,1}$ are algebras and the embedding $B^{1+\frac{1}{p}}_{p,1}\hookrightarrow B^{\frac{1}{p}}_{p,1}\hookrightarrow L^\infty$ holds. Note that the operator $\partial_x{\rm{p}}\ast$ is a $S^{-1}$-mutiplier. Then, we have
\begin{align*}
&\|\partial_x{\rm{p}}\ast\big((u^n)^k\big)\|_{B^{1+\frac{1}{p}}_{p,1}}\leq C\|u^n\|^k_{B^{1+\frac{1}{p}}_{p,1}},\ k=2,3,4,\\
&\|\partial_x{\rm{p}}\ast\big((u^n_x)^2\big)\|_{B^{1+\frac{1}{p}}_{p,1}}\leq C\|u^n_x\|^2_{B^{\frac{1}{p}}_{p,1}}\leq C\|u^n\|^2_{B^{1+\frac{1}{p}}_{p,1}}.
\end{align*}	
So that
\begin{align}
\|G(u^n)\|_{B^{1+\frac{1}{p}}_{p,1}}\leq C\Big(\|u^n\|^2_{B^{1+\frac{1}{p}}_{p,1}}+\|u^n\|^3_{B^{1+\frac{1}{p}}_{p,1}}+\|u^n\|^4_{B^{1+\frac{1}{p}}_{p,1}}\Big),\label{bb}
\end{align}
which implies $G(u^n)\in L^\infty\Big([0,T];B^{1+\frac{1}{p}}_{p,1}\Big)$ for any $T>0$. From Lemma \ref{existence}, we know \eqref{aa} has a global solution $u^{n+1}\in C\Big([0,T];B^{1+\frac{1}{p}}_{p,1}\Big)$.
Thanks to Lemma \ref{priori estimate} and \eqref{bb}, we infer that
\begin{align}
\|u^{n+1}(t)\|_{B^{1+\frac{1}{p}}_{p,1}}
\leq&e^{C\int_0^t\|u^n(t')\|_{B^{1+\frac{1}{p}}_{p,1}}{\ud}t'}\bigg(\|u_0\|_{B^{1+\frac{1}{p}}_{p,1}}\notag\\
&+C\int_0^te^{-C\int_0^\tau\|u^n(\tau')\|_{B^{1+\frac{1}{p}}_{p,1}}{\ud}\tau'}\big(\|u^n\|^2_{B^{1+\frac{1}{p}}_{p,1}}+\|u^n\|^3_{B^{1+\frac{1}{p}}_{p,1}}+\|u^n\|^4_{B^{1+\frac{1}{p}}_{p,1}}\big){\ud}\tau\bigg).\label{dd}
\end{align}
Fix a $T>0$ such that
$1-6CT\Big(\|u_0\|_{B^{1+\frac{1}{p}}_{p,1}}+\|u_0\|^2_{B^{1+\frac{1}{p}}_{p,1}}+\|u_0\|^3_{B^{1+\frac{1}{p}}_{p,1}}\Big)>0$ and suppose by induction that
\begin{align}	\|u^{n}(t)\|_{B^{1+\frac{1}{p}}_{p,1}}
\leq\frac{\|u_0\|_{B^{1+\frac{1}{p}}_{p,1}}}{\Big[1-6Ct\Big(\|u_0\|_{B^{1+\frac{1}{p}}_{p,1}}+\|u_0\|^2_{B^{1+\frac{1}{p}}_{p,1}}+\|u_0\|^3_{B^{1+\frac{1}{p}}_{p,1}}\Big)\Big]^{\frac{1}{3}}},\ \ \ \ \ \ \forall\ t\in[0,T].\label{hh}
\end{align}
Plugging \eqref{hh} into \eqref{dd} yields
\begin{align*}
\|u^{n+1}(t)\|_{B^{1+\frac{1}{p}}_{p,1}}
\leq\frac{\|u_0\|_{B^{1+\frac{1}{p}}_{p,1}}}{\Big[1-6Ct\Big(\|u_0\|_{B^{1+\frac{1}{p}}_{p,1}}+\|u_0\|^2_{B^{1+\frac{1}{p}}_{p,1}}+\|u_0\|^3_{B^{1+\frac{1}{p}}_{p,1}}\Big)\Big]^{\frac{1}{3}}},\ \ \ \ \ \forall\ t\in[0,T].
\end{align*}
Therefore, $\{u^n\}_{n\in\mathbb{N}}$ is uniformly bounded in $C\Big([0,T];B^{1+\frac{1}{p}}_{p,1}\Big)$.\\
	
\textbf{Step 2. We shall use the compactness method for the approximating sequence $\{u^n\}_{n\in\mathbb{N}}$ to get a solution $u$ which satifies Eq.  \eqref{scaling} in the sence of distributions.}

Since $u^n$ is uniformly bounded in $L^\infty\big([0,T];B^{1+\frac{1}{p}}_{p,1}\big)$, we can deduce that
\begin{align*}
&\|u^n\partial_{x}u^n\|_{B^{\frac{1}{p}}_{p,1}}\leq\|u^n\|_{B^{\frac{1}{p}}_{p,1}}\|\partial_{x}u^n\|_{B^{\frac{1}{p}}_{p,1}}\leq\|u^n\|^2_{B^{1+\frac{1}{p}}_{p,1}}\leq C,\\
&\|G(u^{n})\|_{B^{\frac{1}{p}}_{p,1}}\leq C\big(\|u^n\|^2_{B^{1+\frac{1}{p}}_{p,1}}+\|u^n\|^3_{B^{1+\frac{1}{p}}_{p,1}}+\|u^n\|^4_{B^{1+\frac{1}{p}}_{p,1}}\big)\leq C,
\end{align*}	
which infer that $\partial_{t}u^n$ is uniformly bounded in $L^\infty\big([0,T];B^{\frac{1}{p}}_{p,1}\big)$. Thus, 
\begin{align*}
u^n\ \text{is uniformly bounded in}\ C\big([0,T];B^{1+\frac{1}{p}}_{p,1}\big)\cap C^{\frac{1}{2}}\big([0,T];B^{\frac{1}{p}}_{p,1}\big).	
\end{align*}
Let $\{\phi_j\}_{j\in\mathbb{N}}$ be a sequence of smooth functions with value in $[0,1]$ supported in the ball $B(0,j+1)$ and equal to 1 on $B(0,j)$. Notice that the map $z\mapsto\phi_jz$ is compact from $B^{1+\frac{1}{p}}_{p,1}$ to $B^{\frac{1}{p}}_{p,1}$ by Theorem 2.94 in \cite{book}. Taking advantage of the Ascoli's theorem and the Cantor's diagonal process, we can find some function $u_j$ such that for any $j\in\mathbb{N}$, $\phi_ju^n$ tends to $u_j$. From that, we can easily deduce that there exists some function $u$ such that for all $\phi\in\mathcal{D}$, $\phi u^n$ tends to $\phi u$ in $C\big([0,T];B^{\frac{1}{p}}_{p,1}\big)$. Combining the uniformly boundness of $u^n$ and the Fatou property for Besov spaces, we realiy obtain that $u\in L^{\infty}\big([0,T];B^{1+\frac{1}{p}}_{p,1}\big)$. By virtue of the interpolation inequality, we have $\phi u^n$ tends to $\phi u$ in $C\big([0,T];B^{1+\frac{1}{p}-\varepsilon}_{p,1}\big)$ for any $\varepsilon>0$.  Next, it is a routine process to prove that $u$ satifies Eq. \eqref{scaling} in the sence of distributions. Note that the right side of the Eq. \eqref{scaling} belongs to $L^{\infty}\big([0,T];B^{\frac{1}{p}}_{p,1}\big)$. Thus, according to Lemma \ref{existence}, we see that $u\in C\big([0,T];B^{1+\frac{1}{p}}_{p,1}\big)$. Using Eq. \eqref{scaling} again, we get $\partial_{t}u\in C\big([0,T];B^{\frac{1}{p}}_{p,1}\big)$. Hence, $u\in C\big([0,T];B^{1+\frac{1}{p}}_{p,1}\big)\cap C^1\big([0,T];B^{\frac{1}{p}}_{p,1}\big)$.\\
	
\textbf{Step 3. Uniqueness.}
	
We establish the associated Lagrangian scale of \eqref{scaling}, i.e. the following initial valve problem
\begin{equation}\label{ODE}
\left\{\begin{array}{ll}
\frac{{\ud}y}{{\ud}t}=u\big(t,y(t,\xi)\big),\\
y(0,\xi)=\xi.
\end{array}\right.
\end{equation}
\begin{lemm}\label{solution}\cite{c2}
Let $u\in C([0,T];H^s)\cap C^1([0,T];H^{s-1}),\ s\geq2$. Then the problem \eqref{ODE} has a unique solution $y\in C^1([0,T]\times\mathbb{R};\mathbb{R})$. Moreover, the map $y(t,\cdot)$ is an increasing diffeomorphism of $\mathbb{R}$ with
\begin{equation*}
y_{\xi}(t,\xi)=\exp\Big(\int_0^t u_x(\tau,y(\tau,\xi)){\ud}\tau\Big)>0,\ \forall\ (t,\xi)\in[0,T]\times\mathbb{R}.
\end{equation*}
\end{lemm}
\begin{rema}
Lemma \ref{solution} implies that the $L^\infty$ norm of any function $u(t,\cdot)\in L^\infty$ is preserved under the family of diffeomorphisms $y(t,\cdot)$ with any $t\in[0,T]$, that is, $\|u(t,\cdot)\|_{L^\infty}=\|u(t,y(t,\cdot))\|_{L^\infty}$ for any $t\in[0,T]$.
\end{rema}	
Introduce a new variable $U(t,\xi)=u\big(t,y(t,\xi)\big)$. Since $u$ is uniformly bounded in $C\Big([0,T];B^{1+\frac{1}{p}}_{p,1}\Big)\hookrightarrow C\Big([0,T];W^{1,\infty}\cap W^{1,p}\Big)$, we can easily deduce that $y_{\xi}$ is bounded in $L^\infty([0,T];L^\infty)$ by the Gronwall lemma. Hence, $U(t,\xi)$ is bounded in $L^\infty([0,T];W^{1,\infty})$.  	
Owing to \eqref{ODE} and \eqref{scaling}, we can deduce
\begin{align}
y(t,\xi)&=\xi+\int_0^tU{\ud}\tau,\label{y}\\
y_\xi(t,\xi)&=1+\int_0^tU_\xi {\ud}\tau,\label{yxi}\\
U_t(t,\xi)&=\Big(-\partial_{x}{\rm{p}}\ast\big(\frac{1}{2}u^2_x+c_1u^2+c_2u^3+c_3u^4\big)\Big)(t,y(t,\xi))\notag\\
&=\frac{1}{2}\int^{+\infty}_{-\infty}{\rm{sgn}}(y-x)e^{-|y-x|}\big(\frac{1}{2}u^2_x+c_1u^2+c_2u^3+c_3u^4\big){\ud}x.\label{u}
\end{align}
For $t>0$ small enough, we can find two positive constants $C_1,~C_2$ such that $C_1\leq y_\xi\leq C_2$. Without loss of generality, we assume that $t>0$ is sufficiently small, otherwise we can use the continuous method. In this case, we can easily prove that $U(t,\xi)\in L^\infty([0,T];W^{1,p})$. Indeed,
\begin{align*}
&\|U\|_{L^p}^p=\int_{-\infty}^{+\infty}|U(t,\xi)|^p{\ud}\xi=\int_{-\infty}^{+\infty}|u(t,y(t,\xi))|^p\frac{1}{y_\xi}{\ud}y\leq\frac{1}{C_1}\|u\|_{L^p}^p\leq C,\\
&\|U_\xi\|_{L^p}^p=\int_{-\infty}^{+\infty}|U_\xi(t,\xi)|^p{\ud}\xi=\int_{-\infty}^{+\infty}|u_x(t,y(t,\xi))|^p{y_\xi}^{p-1}{\ud}y\leq{C_2}^{p-1}\|u_x\|_{L^p}^p\leq C.
\end{align*}
Thus, we obtain that $U(t,\xi)\in L^\infty([0,T];W^{1,\infty}\cap W^{1,p})$, $y(t,\xi)-\xi\in L^\infty([0,T];W^{1,\infty}\cap W^{1,p})$ and $C_{1}\leq y_\xi\leq C_{2}$ for any $t\in [0,T].$

Now, we prove the uniqueness in Lagrangian coordinates. Suppose that $u_1,u_2$ are two solutions to \eqref{scaling}, so we see for $i=1,2$,
\begin{align}
&\frac{{\ud}}{{\ud}t}U_{i}=\frac{1}{2}\int_{-\infty}^{+\infty}{\rm{sgn}}\big(y_i(t,\xi)-x\big)e^{-|y_i(t,\xi)-x|}\big(c_{1}u_i^2+\frac{1}{2}u_{ix}^2+c_{2}u_{i}^{3}+c_{3}u_{i}^{4}\big){\ud}x,\label{ut}\\
&\frac{{\ud}}{{\ud}t}U_{i\xi}=\Big(c_{1}U_i^2+\frac{1}{2}\frac{U^2_{i\xi}}{y^{2}_{i\xi}}+c_{2}U_i^3+c_{3}U_i^4\Big)y_{i\xi}-\frac{y_{i\xi}}{2}\int_{-\infty}^{+\infty}e^{-|y_i(t,\xi)-x|}\big(c_{1}u_{i}^{2}+\frac{1}{2}u_{ix}^2+c_2u_i^3+c_3u_i^4\big){\ud}x.\label{uxt}
\end{align}
We first estimate $$\int_{-\infty}^{+\infty}{\rm{sgn}}\big(y_1(\xi)-x\big)e^{-|y_1(\xi)-x|}u_{1x}^2{\ud}x-\int_{-\infty}^{+\infty}{\rm{sgn}}\big(y_2(\xi)-x\big)e^{-|y_2(\xi)-x|}u_{2x}^2{\ud}x.$$
Since $y_i$ is monotonically increasing, then ${\rm sgn}\big(y_i(\xi)-y_i(\eta)\big)={\rm sgn}\big(\xi-\eta\big)$ for $i=1,2$. Thus, we have 
\begin{align}
&\int_{-\infty}^{+\infty}{\rm{sgn}}\big(y_1(\xi)-x\big)e^{-|y_1(\xi)-x|}u_{1x}^2{\ud}x-\int_{-\infty}^{+\infty}{\rm{sgn}}\big(y_2(\xi)-x\big)e^{-|y_2(\xi)-x|}u_{2x}^2{\ud}x\notag\\
=&\int_{-\infty}^{+\infty}{\rm{sgn}}\big(y_1(\xi)-y_1(\eta)\big)e^{-|y_1(\xi)-y_1(\eta)|}\frac{U_{1\eta}^2}{y_{1\eta}}{\ud}\eta-\int_{-\infty}^{+\infty}{\rm{sgn}}\big(y_2(\xi)-y_2(\eta)\big)e^{-|y_2(\xi)-y_2(\eta)|}\frac{U_{2\eta}^2}{y_{2\eta}}{\ud}\eta\notag\\
=&\int_{-\infty}^{+\infty}{\rm{sgn}}\big(\xi-\eta\big)e^{-|y_1(\xi)-y_1(\eta)|}\frac{U_{1\eta}^2}{y_{1\eta}}{\ud}\eta-\int_{-\infty}^{+\infty}{\rm{sgn}}\big(\xi-\eta\big)e^{-|y_2(\xi)-y_2(\eta)|}\frac{U_{2\eta}^2}{y_{2\eta}}{\ud}\eta\notag\\
=&\int_{-\infty}^{+\infty}{\rm{sgn}}\big(\xi-\eta\big)\left(e^{-|y_1(\xi)-y_1(\eta)|}-e^{-|y_2(\xi)-y_2(\eta)|}\right)\frac{U_{1\eta}^2}{y_{1\eta}}{\ud}\eta\notag\\
&+\int_{-\infty}^{+\infty}{\rm{sgn}}\big(\xi-\eta\big)e^{-|y_2(\xi)-y_2(\eta)|}\left(\frac{U_{2\eta}^2}{y_{2\eta}}-\frac{U_{1\eta}^2}{y_{1\eta}}\right){\ud}\eta\notag\\
:=&J_1+J_2.\label{u1}
\end{align}
If $\xi>\eta(\text{or}\ \xi<\eta)$, then $y_i(\xi)>y_i(\eta)\big(\text{or}\ y_i(\xi)<y_i(\eta)\big)$, and we can deduce
\begin{align}
J_1=&\int_{-\infty}^{\xi}\left(e^{-(y_1(\xi)-y_1(\eta))}-e^{-(y_2(\xi)-y_2(\eta))}\right)\frac{U_{1\eta}^2}{y_{1\eta}}{\ud}\eta-\int_{\xi}^{+\infty}\left(e^{y_1(\xi)-y_1(\eta)}-e^{y_2(\xi)-y_2(\eta)}\right)\frac{U_{1\eta}^2}{y_{1\eta}}{\ud}\eta\notag\\
=&\int_{-\infty}^{\xi}e^{-(\xi-\eta)}\left(e^{-\int_0^t(U_1(\xi)-U_1(\eta)){\ud}\tau}-e^{-\int_0^t(U_2(\xi)-U_2(\eta)){\ud}\tau}\right)\frac{U_{1\eta}^2}{y_{1\eta}}{\ud}\eta\notag\\
&-\int_{\xi}^{+\infty}e^{\xi-\eta}\left(e^{\int_0^tU_1(\xi)-U_1(\eta){\ud}\tau}-e^{\int_0^tU_2(\xi)-U_2(\eta){\ud}\tau}\right)\frac{U_{1\eta}^2}{y_{1\eta}}{\ud}\eta\notag\\
\leq&C\|U_1-U_2\|_{L^\infty}\Big[\int_{-\infty}^{\xi}e^{-(\xi-\eta)}\frac{U_{1\eta}^2}{y_{1\eta}}{\ud}\eta+\int_{\xi}^{+\infty}e^{\xi-\eta}\frac{U_{1\eta}^2}{y_{1\eta}}{\ud}\eta\Big]\notag\\
\leq&C\|U_1-U_2\|_{L^\infty}\Big[1_{\geq0}(x)e^{-|x|}\ast\frac{U_{1\eta}^2}{y_{1\eta}}+1_{\leq0}(x)e^{-|x|}\ast\frac{U_{1\eta}^2}{y_{1\eta}}\Big].\label{i1}
\end{align}
In the same way, we have
\begin{align}
J_2\leq C\Big[1_{\geq0}(x)e^{-|x|}\ast\Big(|U_{1\eta}-U_{2\eta}|+|y_{1\eta}-y_{2\eta}|\Big)+1_{\leq0}(x)e^{-|x|}\ast\Big(|U_{1\eta}-U_{2\eta}|+|y_{1\eta}-y_{2\eta}|\Big)\Big].\label{i22}
\end{align}
Combining with \eqref{u1}-\eqref{i22}, we find that
\begin{align}
&\int_{-\infty}^{+\infty}{\rm{sgn}}\big(y_1(\xi)-x\big)e^{-|y_1(\xi)-x|}u_{1x}^2{\ud}x-\int_{-\infty}^{+\infty}{\rm{sgn}}\big(y_2(\xi)-x\big)e^{-|y_2(\xi)-x|}u_{2x}^2{\ud}x\notag\\
\leq&C\|U_1-U_2\|_{L^\infty}\Big[1_{\geq0}(x)e^{-|x|}\ast\frac{U_{1\eta}^2}{y_{1\eta}}+1_{\leq0}(x)e^{-|x|}\ast\frac{U_{1\eta}^2}{y_{1\eta}}\Big]\notag\\
&+C\Big[1_{\geq0}(x)e^{-|x|}\ast\Big(|U_{1\eta}-U_{2\eta}|+|y_{1\eta}-y_{2\eta}|\Big)+1_{\leq0}(x)e^{-|x|}\ast\Big(|U_{1\eta}-U_{2\eta}|+|y_{1\eta}-y_{2\eta}|\Big)\Big].\label{i12}
\end{align}
Similar to \eqref{i12}, we get that for $k=2,3,4,$
\begin{align}
&\int_{-\infty}^{+\infty}{\rm{sgn}}\big(y_1(\xi)-x\big)e^{-|y_1(\xi)-x|}u_{1}^k{\ud}x-\int_{-\infty}^{+\infty}{\rm{sgn}}\big(y_2(\xi)-x\big)e^{-|y_2(\xi)-x|}u_{2}^k{\ud}x\notag\\
\leq&C\|U_1-U_2\|_{L^\infty}\Big[1_{\geq0}(x)e^{-|x|}\ast\Big( U_{1\eta}^ky_{1\eta}\Big)+1_{\leq0}(x)e^{-|x|}\ast\Big( U_{1\eta}^ky_{1\eta}\Big)\Big]\notag\\
&+C\Big[1_{\geq0}(x)e^{-|x|}\ast\Big(|U_{1}-U_{2}|+|y_{1\eta}-y_{2\eta}|\Big)+1_{\leq0}(x)e^{-|x|}\ast\Big(|U_{1}-U_{2}|+|y_{1\eta}-y_{2\eta}|\Big)\Big].\label{j12}
\end{align}
It follows from \eqref{ut}, \eqref{i12} and \eqref{j12} that 
\begin{align*}
|U_1(t)-U_2(t)|\leq&|U_1(0)-U_2(0)|+C\int_0^t\Big[1_{\geq0}(x)e^{-|x|}\ast\Big(|U_{1}-U_{2}|+|U_{1\eta}-U_{2\eta}|+|y_{1\eta}-y_{2\eta}|\Big)\Big]\\
&+\Big[1_{\leq0}(x)e^{-|x|}\ast\Big(|U_{1}-U_{2}|+|U_{1\eta}-U_{2\eta}|+|y_{1\eta}-y_{2\eta}|\Big)\Big]\\
&+\|U_1-U_2\|_{L^\infty}\Big[1_{\geq0}(x)e^{-|x|}\ast\Big(\frac{U_{1\eta}^2}{y_{1\eta}}+U_{1\eta}^2y_{1\eta}+U_{1\eta}^3y_{1\eta}+U_{1\eta}^4y_{1\eta}\Big)\\
&\qquad\qquad\qquad\quad+1_{\leq0}(x)e^{-|x|}\ast\Big(\frac{U_{1\eta}^2}{y_{1\eta}}+U_{1\eta}^2y_{1\eta}+U_{1\eta}^3y_{1\eta}+U_{1\eta}^4y_{1\eta}\Big)\Big]{\ud}\tau
\end{align*}
which infers that
\begin{align}
\|U_1-U_2\|_{L^\infty\cap L^p}\leq&\|U_1(0)-U_2(0)\|_{L^\infty\cap L^p}\notag\\
&+C\int_0^t\|U_1-U_2\|_{L^\infty\cap L^p}+\|U_{1\xi}-U_{2\xi}\|_{L^\infty\cap L^p}+\|y_{1\xi}-y_{2\xi}\|_{L^\infty\cap L^p}{\ud}\tau.\label{U12}
\end{align}
Noticing that $U_{i\xi}(i=1,2)$ satifies \eqref{uxt}, we can similarly obtain
\begin{align}
\|U_{1\xi}-U_{2\xi}\|_{L^\infty\cap L^p}\leq&\|U_{1\xi}(0)-U_{2\xi}(0)\|_{L^\infty\cap L^p}\notag\\
&+C\int_0^t\|U_1-U_2\|_{L^\infty\cap L^p}+\|U_{1\xi}-U_{2\xi}\|_{L^\infty\cap L^p}+\|y_{1\xi}-y_{2\xi}\|_{L^\infty\cap L^p}{\ud}\tau.\label{Ux12}
\end{align}
Since $y_i$ satisfies \eqref{y}--\eqref{yxi}, it follows that
\begin{align}
&\|y_{1}-y_{2}\|_{L^\infty\cap L^p}\leq C\int_0^t\|U_{1}-U_{2}\|_{L^\infty\cap L^p}{\ud}\tau,\label{y12}\\
&\|y_{1\xi}-y_{2\xi}\|_{L^\infty\cap L^p}\leq C\int_0^t\|U_{1\xi}-U_{2\xi}\|_{L^\infty\cap L^p}{\ud}\tau.\label{yxi12}	
\end{align}
Inequalities \eqref{U12}-\eqref{yxi12} and Grownall lemma yield
\begin{align*}
\|U_1-U_2\|_{W^{1,\infty}\cap W^{1,p}}+\|y_{1}-y_{2}\|_{W^{1,\infty}\cap W^{1,p}}\leq C\|U_1(0)-U_2(0)\|_{W^{1,\infty}\cap W^{1,p}}\leq C\|u_1(0)-u_2(0)\|_{B^{1+\frac{1}{p}}_{p,1}}.	
\end{align*}
Returning to the original coordinates, we see that
\begin{align*}
\|u_1-u_2\|_{L^p}\leq& C\|u_1\circ y_1-u_2\circ y_1\|_{L^p}\\
\leq& C\|u_1\circ y_1-u_2\circ y_2+u_2\circ y_2-u_2\circ y_1\|_{L^p}\\
\leq& C\|U_1-U_2\|_{L^p}+C\|u_{2x}\|_{L^\infty}\|y_1-y_2\|_{L^p}\\
\leq& C\|u_1(0)-u_2(0)\|_{B^{1+\frac{1}{p}}_{p,1}}.
\end{align*}
So if $u_1(0)=u_2(0)$, we can  immediately obtain the uniqueness.\\
	
\textbf{Step 4. The continuous dependence.}
	
Let $u^n_0\rightarrow u^{\infty}_0$ in $B^{1+\frac{1}{p}}_{p,1}$ for $n\in\mathbb{N}$. Then we have $\partial_{x}u^n_0\rightarrow \partial_{x}u^{\infty}_0$ in $B^{\frac{1}{p}}_{p,1}$. For $n\in\overline{\mathbb{N}}:=\mathbb{N}\cup{\infty}$, denote by $u^n$ the solution to \eqref{scaling} with initial value $u^n_0$. By \textbf{Step 1--Step 2}, we know that for $n\in\overline{\mathbb{N}}$
\begin{align*}
u^n\in C\big([0,T];B^{1+\frac{1}{p}}_{p,1}\big),\quad\partial_{x}u^n\in C\big([0,T];B^{\frac{1}{p}}_{p,1}\big)\quad \text{and}\quad \|u^n\|_{L^\infty([0,T];B^{1+\frac{1}{p}}_{p,1})}\leq C\|u^n_0\|_{B^{1+\frac{1}{p}}_{p,1}}.
\end{align*}
Owing to \textbf{Step 3}, we have for all $t\in[0,T]$,
\begin{align*}
\|\big(u^n-u^\infty\big)(t)\|_{L^p}\leq C\|u_0^n-u_0^\infty\|_{B^{1+\frac{1}{p}}_{p,1}}.
\end{align*}
By the embedding inequality $L^p\hookrightarrow B^0_{p,\infty}$, we get for all $t\in[0,T]$,
\begin{align*}
\|\big(u^n-u^\infty\big)(t)\|_{B^0_{p,\infty}}\leq\|\big(u^n-u^\infty\big)(t)\|_{L^p}\leq C\|u_0^n-u_0^\infty\|_{B^{1+\frac{1}{p}}_{p,1}},	
\end{align*}
which implies that $u^n$ tends to $u^\infty$ in $C\big([0,T];B^{0}_{p,\infty}\big)$. Taking advantage of the interpolation inequality, we see that $u^n\rightarrow u^\infty$ in $C\big([0,T];B^{1+\frac{1}{p}-\epsilon}_{p,1}\big)$ for any $\epsilon>0$. Choosing $\epsilon=1$, we find
\begin{align}
u^n\rightarrow u^\infty\qquad\text{in}\quad C\big([0,T];B^{\frac{1}{p}}_{p,1}\big).\label{untend}
\end{align}
Next, we just have to prove $\partial_{x}u^n\rightarrow \partial_{x}u^\infty$ in $C\big([0,T];B^{\frac{1}{p}}_{p,1}\big)$. Before that, let's present a useful lemma.
\begin{lemm}\label{continuous}\cite{book,liy}
Define $\overline{\mathbb{N}}=\mathbb{N}\cup\{\infty\}$. Let $F\in L^\infty\big([0,T];B^{\frac{1}{p}}_{p,1}\big)$, $a_0\in B^{\frac{1}{p}}_{p,1}$ and  $A^n\rightarrow A^\infty$ in $L^1\big([0,T];B^{\frac{1}{p}}_{p,1}\big)$. For $n\in\overline{\mathbb{N}}$, denote by $a^n$ the solution of 
\begin{equation}\label{an}
\left\{\begin{array}{l}
\partial_{t}a^n+A^n\partial_{x}a^n=F,\\
a^n(0,x)=a_0(x).
\end{array}\right.
\end{equation} 	
Then the sequence $\{a^n\}_{n\in\mathbb{N}}$ converges to $a^\infty$ in $C\big([0,T];B^{\frac{1}{p}}_{p,1}\big)$.\end{lemm}
We continue proving $\partial_{x}u^n\rightarrow \partial_{x}u^\infty$ in $C\big([0,T];B^{\frac{1}{p}}_{p,1}\big)$. For simplicity, let $v^n=\partial_{x}u^n,\ v^\infty=\partial_{x}u^\infty$. Split $v^n$ into $w^n+z^n$ with $(w^n,z^n)$ satisfying 
\begin{equation*}
\left\{\begin{array}{l}
\partial_{t}w^n+u^n\partial_{x}w^n=F^\infty,\\
w^n(0,x)=v^\infty_0=\partial_{x}u^\infty_0
\end{array}\right.
\end{equation*} 
and 
\begin{equation*}
\left\{\begin{array}{l}
\partial_{t}z^n+u^n\partial_{x}z^n=F^n-F^\infty, \\
z^n(0,x)=v^n_0-v^\infty_0=\partial_{x}u^n_0-\partial_{x}u^\infty_0
\end{array}\right.
\end{equation*}
where 
\begin{align*}
&F^n=c_1(u^n)^2-\frac{1}{2}(u_x^n)^2+c_2(u^n)^3+c_3(u^n)^4-{\rm{p}}\ast\Big(c_1(u^n)^2+\frac{1}{2}(u_x^n)^2+c_2(u^n)^3+c_3(u^n)^4\Big),\\
&F^\infty=c_1(u^\infty)^2-\frac{1}{2}(u_x^\infty)^2+c_2(u^\infty)^3+c_3(u^\infty)^4-{\rm{p}}\ast\Big(c_1(u^\infty)^2+\frac{1}{2}(u_x^\infty)^2+c_2(u^\infty)^3+c_3(u^\infty)^4\Big).
\end{align*}
Because $\{u^n\}_{n\in\overline{\mathbb{N}}}$ is bounded in $C\big([0,T];B^{1+\frac{1}{p}}_{p,1}\big)$, then $\{\partial_{x}u^n\}_{n\in\overline{\mathbb{N}}}$ and $\{F^n\}_{n\in\overline{\mathbb{N}}}$ are bounded in $C\big([0,T];B^{\frac{1}{p}}_{p,1}\big)$. Notice again that $u^n\rightarrow u^\infty$ in $C\big([0,T];B^{\frac{1}{p}}_{p,1}\big)$. Lemma \ref{continuous} thus ensures that
\begin{align}
w^n\rightarrow w^\infty\qquad \text{in}\quad  C\big([0,T];B^{\frac{1}{p}}_{p,1}\big).\label{winfty}
\end{align} 
Next, according to Lemma \ref{existence}--\ref{priori estimate}, for any $t\in[0,T]$, we obtain $z^\infty=0$. Noting that the operator ${\rm{p}}\ast$ is a $S^{-2}$-mutiplier and  $\{u^n\}_{n\in\overline{\mathbb{N}}}$ is bounded in $L^{\infty}\big([0,T];B^{1+\frac{1}{p}}_{p,1}\big)$, we have 
\begin{align*}
\|F^n-F^\infty\|_{B^{\frac{1}{p}}_{p,1}}\leq C\left(\|u^n-u^\infty\|_{B^{\frac{1}{p}}_{p,1}}+\|v^n-v^\infty\|_{B^{\frac{1}{p}}_{p,1}}\right).
\end{align*}
It follows that for all $n\in\mathbb{N}$,
\begin{align}
\|z^n(t)\|_{B^{\frac{1}{p}}_{p,1}}
\leq& e^{\int_0^t\|\partial_{x}u^n(t')\|_{B^{\frac{1}{p}}_{p,1}}{\ud}t'}\left(\|v^n_0-v^\infty_0\|_{B^{\frac{1}{p}}_{p,1}}+\int_0^t\|F^n-F^\infty\|_{B^{\frac{1}{p}}_{p,1}}{\ud}\tau\right)\notag\\
\leq&C\left(\|v^n_0-v^\infty_0\|_{B^{\frac{1}{p}}_{p,1}}+\int_0^t\|u^n-u^\infty\|_{B^{\frac{1}{p}}_{p,1}}+\|w^n-w^\infty\|_{B^{\frac{1}{p}}_{p,1}}+\|z^n\|_{B^{\frac{1}{p}}_{p,1}}{\ud}\tau\right).\label{zn}
\end{align}
Using the facts that
\begin{itemize}
\item [-]~~$v_0^n$ tends to $v_0^{\infty}$ in $B^{\frac{1}{p}}_{p,1}$;
\item [-]~~$u^n$ tends to $u^{\infty}$ in $C\big([0,T];B^{\frac{1}{p}}_{p,1}\big)$;
\item [-]~~$w^n$ tends to $w^{\infty}$ in $C\big([0,T];B^{\frac{1}{p}}_{p,1}\big)$,
\end{itemize}
and then applying the Gronwall lemma, we conclude that $z^n$ tends to $0$ in $C\big([0,T];B^{\frac{1}{p}}_{p,1}\big)$.\\
Hence,
\begin{align*}
\|v^n-v^\infty\|_{L^\infty\big([0,T];B^{\frac{1}{p}}_{p,1}\big)}\leq&\|w^n-w^\infty\|_{L^\infty\big([0,T];B^{\frac{1}{p}}_{p,1}\big)}+\|z^n-z^\infty\|_{L^\infty\big([0,T];B^{\frac{1}{p}}_{p,1}\big)}\\
\leq&\|w^n-w^\infty\|_{L^\infty\big([0,T];B^{\frac{1}{p}}_{p,1}\big)}+\|z^n\|_{L^\infty\big([0,T];B^{\frac{1}{p}}_{p,1}\big)}\qquad\qquad\rightarrow 0\quad\text{as}\ n\rightarrow\infty,
\end{align*}
that is 
\begin{align*}
\partial_{x}u^n\rightarrow \partial_{x}u^\infty\qquad\text{in}\quad C\big([0,T];B^{\frac{1}{p}}_{p,1}\big).
\end{align*}
Consequently, we prove the continuous dependence.
	
In conclusion, combining with \textbf{Step 1--Step 4}, we complete the proof of Theorem \ref{th}.
\end{proof}

\section{Non-uniform continuous dependence in both supercritical and critical Besov spaces}
\par
In this section, we investigate the non-uniform continuous dependence  of the Cauchy problem for Eq. \eqref{scaling} in Besov space $B^{s}_{p,r},\ s>\max\{\frac{3}{2},1+\frac{1}{p}\},1\leq p\leq+\infty,\ 1\leq r<+\infty$ or $s=1+\frac{1}{p},\ 1\leq p\leq2,\ r=1$.

Before that, we introduce smooth, radial cut-off functions in frequency space. Let $\hat{\psi}\in C_0^\infty(\mathbb{R})$ be an even, real-valued and non-negative function on $\mathcal{D}$ and satisfy
\begin{equation}\label{psi}
\hat{\psi}(\xi)=\left\{\begin{array}{ll}
		1,\quad if\ |\xi|\leq\frac{1}{4,}\\
		0,\quad if\ |\xi|\geq\frac{1}{2}.
	\end{array}\right.
\end{equation}
From Fourier inversion formula, we can easily deduce that 
\begin{align*}
&\psi(x)=\frac{1}{2\pi}\int\limits_{\mathbb{R}}e^{ix\xi}\hat{\psi}(\xi){\ud}\xi;\\
&\partial_{x}\psi(x)=\frac{1}{2\pi}\int\limits_{\mathbb{R}}i\xi e^{ix\xi}\hat{\psi}(\xi){\ud}\xi,
\end{align*}
which implies by the Fubini theorem that
\begin{align*}
&\|\psi\|_{L^\infty}=\sup\limits_{x\in\mathbb{R}}\frac{1}{2\pi}\bigg|\int\limits_{\mathbb{R}}\cos(x\xi)\hat{\psi}(\xi){\ud}\xi\bigg|\leq\frac{1}{2\pi}\int\limits_{\mathbb{R}}\hat{\psi}(\xi){\ud}\xi\leq C,\\
&\|\partial_{x}\psi\|_{L^\infty}\leq\frac{1}{2\pi}\int\limits_{\mathbb{R}}|\xi|\hat{\psi}(\xi){\ud}\xi\leq C,\\
&\psi(0)=\frac{1}{2\pi}\int\limits_{\mathbb{R}}\hat{\psi}(\xi){\ud}\xi.
\end{align*}

Here is the crucial lemmas that we will use later.
\begin{lemm}\label{bspr}
Let $p,\ r\in[1,+\infty]$ and $s>1$. Let $u\in L^{\infty}([0,T];B^{s}_{p,r})$ solve \eqref{scaling} with initial data $u_{0}\in B^{s}_{p,r}$. There exist constants $C, C'$ such that for all $t\in[0,T]$, we have
\begin{align*}
&\|u(t)\|_{B^{s}_{p,r}}\leq\|u_{0}\|_{B^{s}_{p,r}}e^{C\int_{0}^{t}\|u_{x}\|_{L^{\infty}}+\|u\|_{L^{\infty}}+\|u\|_{L^{\infty}}^{2}+\|u\|_{L^{\infty}}^{3}{\ud}\tau},\\
&\|u_{x}\|_{L^{\infty}}+\|u\|_{L^{\infty}}+\|u\|_{L^{\infty}}^{2}+\|u\|_{L^{\infty}}^{3}\leq\Big(\|u_{0x}\|_{L^{\infty}}+\|u_{0}\|_{L^{\infty}}+\|u_{0}\|_{L^{\infty}}^{2}+\|u_{0}\|_{L^{\infty}}^{3}\Big)e^{C'\int_{0}^{t}\|u_{x}\|_{L^{\infty}}{\ud}\tau}.
\end{align*}
\end{lemm}
\begin{proof}
The proof is similar to that of Lemma 3.26 in \cite{book}, and here we omit it.
\end{proof}
\begin{lemm}\label{psii}
Let $1\leq a\leq\infty$. Then there is a constant $A>0$ such that 
\begin{align}\label{psiii}
\liminf\limits_{n\rightarrow\infty}\bigg\|\psi^2(\cdot)\cos\Big(\frac{33}{24}2^n\cdot\Big)\bigg\|_{L^a}\geq A.
\end{align}
\end{lemm}
\begin{lemm}\label{wsin}
Let $s\in\mathbb{R}$ and $1\leq p\leq\infty,\ 1\leq r<\infty$. Define the high frequency function ${\rm w}^n_0$ by 
\begin{align*}
{\rm w}^n_0(x)=2^{-ns}\psi(x)\sin\Big(\frac{33}{24}2^{n}x\Big),\qquad n\gg 1.
\end{align*}
Then for any $\theta\in\mathbb{R}$, we have 
\begin{align}
&\|{\rm w}^n_0\|_{L^{p}}\leq C2^{-ns}\|\psi\|_{L^{p}}\leq C2^{-ns},\qquad
\|\partial_{x}{\rm w}^n_0\|_{L^{p}}\leq C2^{-ns+n},\label{wn0}\\	
&\|{\rm w}^n_0\|_{B^\theta_{p,r}}\leq C2^{n(\theta-s)}\|\psi\|_{L^p}\leq C2^{n(\theta-s)}.\label{high}
\end{align}	
\end{lemm}
\begin{lemm}\label{vcos}
Let $s\in\mathbb{R}$ and $1\leq p\leq\infty,\ 1\leq r<\infty$. Define the low frequency function ${\rm v}^n_0$ by 
\begin{align*}
	{\rm v}^n_0(x)=\frac{24}{33}2^{-n}\psi(x),\qquad n\gg 1.
\end{align*}
Then 
\begin{align}
&\|{\rm v}^n_0\|_{L^{p}}\leq C2^{-n}\|\psi\|_{L^{p}}\leq C2^{-n},\quad
\|\partial_{x}{\rm v}^n_0\|_{L^{p}}\leq C2^{-n}\|\partial_{x}\psi\|_{L^{p}}\leq C2^{-n},\label{vn0}\\
&\|{\rm v}^n_0\|_{B^{s}_{p,r}}\leq C2^{-s}\|{\rm v}^n_0\|_{L^{p}}\leq C2^{-n-s},\label{vn0s}	
\end{align}
and there is a constant $\tilde{A}>0$ such that  
\begin{align}\label{low}
\liminf\limits_{n\rightarrow\infty}\big\|{\rm v}^n_0\partial_{x}{\rm w}^n_0\big\|_{B^s_{p,\infty}}\geq \tilde{A}.
\end{align}
\end{lemm}
The above lemmas can be proved by a similar way as Lemmas 3.2--3.4 in \cite{lyz} and here we omit it.\\

Define ${\rm w}^n$ by the solution of Eq. \eqref{scaling} with the initial data ${\rm w}^n_0$. Then we have the following estimates.
\begin{prop}\label{wnw0}
Let $s\in\mathbb{R},\ p,\ r\in[1,\infty]$ and let $(s,p,r)$ meet the condition \eqref{cond}.	
Then for $k=-1,1$, we have
\begin{align}
&\|{\rm w}^n\|_{B^{s+k}_{p,r}}\leq 2^{kn},\label{wn}\\
&\|{\rm w}^n-{\rm w}^n_0\|_{B^{s}_{p,r}}\leq C 2^{-\frac{1}{2}n(s-\frac 3 2)}.\label{w}
\end{align} 
\end{prop}
\begin{proof}
\eqref{high} implies that
\begin{align}
\|{\rm w}^n_0\|_{B^{s+k}_{p,r}}\leq C2^{kn},\ k=-1,0,1.\label{wn-10}	
\end{align}
From Theorem \ref{th}, we know that there exists a $T=T(\|{\rm w}^n_0\|_{B^s_{p,r}})$ such that \eqref{scaling} with initial data ${\rm w}^n_0$ has a unique solution ${\rm w}^n\in E^s_{p,r}(T)$ and $T\approx 1$. Moreover, we have from \eqref{wn-10}
\begin{align}
\|{\rm w}^n\|_{L^{\infty}\big([0,T];B^s_{p,r}\big)}\leq \|{\rm w}^n_0\|_{B^s_{p,r}}\leq C.\label{wnbound}
\end{align}
Similar to \eqref{dd}, we get that for $k=\pm 1$,
\begin{align}
\|{\rm w}^{n}(t)\|_{B^{s+k}_{p,r}}
\leq  C\Bigg(\|{\rm w}^n_0\|_{B^{s+k}_{p,r}}+\int_0^t\|\partial_{x}{\rm{p}}\ast\big(({\rm w}^n)^2+({\rm w}^n_x)^2+({\rm w}^n)^3+({\rm w}^n)^4\big)\|_{B^{s+k}_{p,r}}{\ud}\tau\Bigg).\label{wnk}
\end{align}
Note that $\|\partial_{x}{\rm{p}}\ast\big(({\rm w}^n)^2+({\rm w}^n_x)^2+({\rm w}^n)^3+({\rm w}^n)^4\big)\|_{B^{s+k}_{p,r}}\leq C\|({\rm w}^n)^2+({\rm w}^n_x)^2+({\rm w}^n)^3+({\rm w}^n)^4\|_{B^{s+k-1}_{p,r}}$. For $k=-1$, by use of Lemma \ref{product}, we deduce
\begin{align}
&\|({\rm w}^n)^l\|_{B^{s-2}_{p,r}}\leq\|{\rm w}^n\|_{B^{s-2}_{p,r}}\|{\rm w}^n\|^{l-1}_{B^{s-1}_{p,r}}\leq\|{\rm w}^n\|_{B^{s-1}_{p,r}}\|({\rm w}^n)\|^{l-1}_{B^{s}_{p,r}}=\|{\rm w}^n\|_{B^{s+k}_{p,r}}\|{\rm w}^n\|^{l-1}_{B^{s}_{p,r}},\quad l=2,3,4,\notag\\
&\|({\rm w}^n_x)^2\|_{B^{s-2}_{p,r}}\leq\|{\rm w}^n_x\|_{B^{s-2}_{p,r}}\|{\rm w}^n_x\|_{B^{s-1}_{p,r}}\leq\|{\rm w}^n\|_{B^{s-1}_{p,r}}\|{\rm w}^n\|_{B^{s}_{p,r}}=\|{\rm w}^n\|_{B^{s+k}_{p,r}}\|({\rm w}^n)\|_{B^{s}_{p,r}}.\label{wn-1}
\end{align}
In a similar way, for $k=1$, we obtain
\begin{align}
&\|({\rm w}^n_x)^2\|_{B^{s}_{p,r}}\leq\|{\rm w}^n_x\|_{B^{s}_{p,r}}\|{\rm w}^n_x\|_{L^\infty}\leq\|{\rm w}^n\|_{B^{s+1}_{p,r}}\|{\rm w}^n_x\|_{B^{s-1}_{p,r}}\leq\|{\rm w}^n\|_{B^{s+1}_{p,r}}\|{\rm w}^n\|_{B^{s}_{p,r}}=\|{\rm w}^n\|_{B^{s+k}_{p,r}}\|{\rm w}^n\|_{B^{s}_{p,r}};\notag\\
&\|({\rm w}^n)^l\|_{B^{s}_{p,r}}\leq\|{\rm w}^n\|_{B^{s}_{p,r}}\|{\rm w}^n\|^{l-1}_{B^{s}_{p,r}}\leq\|{\rm w}^n\|_{B^{s+1}_{p,r}}\|{\rm w}^n\|^{l-1}_{B^{s}_{p,r}}=\|{\rm w}^n\|_{B^{s+k}_{p,r}}\|{\rm w}^n\|^{l-1}_{B^{s}_{p,r}},\ l=2,3,4.\label{wn1}
\end{align}
Plugging \eqref{wn-1}--\eqref{wn1} into \eqref{wnk}, we see for $k\pm1$,
\begin{align*}
	\|{\rm w}^{n}(t)\|_{B^{s+k}_{p,r}}
	\leq  C\|{\rm w}^n_0\|_{B^{s+k}_{p,r}}+C\int_0^t\|{\rm w}^n\|_{B^{s+k}_{p,r}}\Big(\|{\rm w}^n\|_{B^{s}_{p,r}}+\|{\rm w}^n\|^2_{B^{s}_{p,r}}+\|{\rm w}^n\|^3_{B^{s}_{p,r}}\Big){\ud}\tau.
\end{align*}
Combining the Gronwall lemma, \eqref{unformly} and \eqref{wn-10}, we find for all $t\in[0,T]$,
\begin{align}
\|{\rm w}^n(t)\|_{B^{s-1}_{p,r}}\leq C2^{-n}\qquad\text{and}\qquad
\|{\rm w}^n(t)\|_{B^{s+1}_{p,r}}\leq C2^{n}.\label{wn-11}	
\end{align}
Set $\delta={\rm w}^n-{\rm w}^n_0$. Then $\delta$ solves the following problem
\begin{equation*} \left\{\begin{array}{ll}
\partial_{t}\delta+{\rm w}^n\partial_{x}\delta=&-\delta\partial_{x}{\rm w}^n_0-{\rm{p}}_x\ast\Big[\frac{1}{2}\big({\rm w}^n_x+({\rm w}^n_0)_x\big)\delta_x+c_1\big({\rm w}^n+{\rm w}^n_0\big)\delta+c_2({\rm w}^n)^2\delta+c_2{\rm w}^n_0\big({\rm w}^n+{\rm w}^n_0\big)\delta\\
&+c_3\big(({\rm w}^n)^2+({\rm w}^n_0)^2\big)\big({\rm w}^n+{\rm w}^n_0\big)\delta\Big]-{\rm{p}}_x\ast\Big[\frac{1}{2}\big(\partial_{x}{\rm w}^n_0\big)^2+c_1({\rm w}^n_0)^2+c_2({\rm w}^n_0)^3+c_3({\rm w}^n_0)^4\Big]\\
&-{\rm w}^n_0\partial_{x}{\rm w}^n_0,  \\
\delta(0,x)=0.
\end{array}\right.
\end{equation*}
\eqref{wn0} and \eqref{high} imply that
\begin{align*}
\|{\rm w}^n_0\partial_{x}{\rm w}^n_0\|_{B^{s-1}_{p,r}}\leq&C\|{\rm w}^n_0\|_{L^\infty}\|\partial_{x}{\rm w}^n_0\|_{B^{s-1}_{p,r}}+\|{\rm w}^n_0\|_{B^{s-1}_{p,r}}\|\partial_{x}{\rm w}^n_0\|_{L^\infty}\\
\leq&C2^{-ns}+2^{-n}2^{-ns+n}\leq C2^{-ns},\\
\|{\rm{p}}_x\ast\big((\partial_{x}{\rm w}^n_0)^2\big)\|_{B^{s-1}_{p,r}}\leq&C\|{\rm{p}}_x\ast\big((\partial_{x}{\rm w}^n_0)^2\big)\|_{B^{s-\frac{1}{2}}_{p,r}}\leq C\|(\partial_{x}{\rm w}^n_0)^2\|_{B^{s-\frac{3}{2}}_{p,r}}
\leq C\|\partial_{x}{\rm w}^n_0\|_{B^{s-\frac{3}{2}}_{p,r}}\|\partial_{x}{\rm w}^n_0\|_{L^\infty}\\
\leq&C2^{n(s-\frac{3}{2})}\|\partial_{x}{\rm w}^n_0\|_{L^p}\|\partial_{x}{\rm w}^n_0\|_{L^\infty}\leq C2^{n(s-\frac{3}{2})}2^{-ns+n}2^{-ns+n}\leq C2^{n(\frac{1}{2}-s)},\\
\|{\rm{p}}_x\ast\big(({\rm w}^n_0)^l\big)\|_{B^{s-1}_{p,r}}\leq&C\|{\rm{p}}_x\ast\big(({\rm w}^n_0)^l\big)\|_{B^{s-\frac 1 2}_{p,r}}
\leq C\|({\rm w}^n_0)^l\|_{B^{s-\frac 3 2}_{p,r}}\leq C\|{\rm w}^n_0\|_{B^{s-\frac 3 2}_{p,r}}\|{\rm w}^n_0\|^{l-1}_{L^\infty}\\
\leq&C2^{n(s-\frac{3}{2})}\|{\rm w}^n_0\|_{L^p}\|{\rm w}^n_0\|^{l-1}_{L^\infty}\leq C2^{n(s-\frac{3}{2})}2^{-nls}\\
\leq& C2^{-n(l-1)s-\frac{3}{2}n}\leq C2^{n(\frac{1}{2}-s)},\qquad l=2,3,4.
\end{align*}
By means of Proposition \ref{Besov} and Lemma \ref{product}, we have
\begin{align*}
&\|\delta\partial_{x}{\rm w}^n_0\|_{B^{s-1}_{p,r}}\leq C\|\delta\|_{B^{s-1}_{p,r}}\|{\rm w}^n_0\|_{B^{s}_{p,r}},\\
&\|{\rm{p}}_x\ast\big(\big({\rm w}^n_x+({\rm w}^n_0)_x\big)\delta_x\big)\|_{B^{s-1}_{p,r}}\leq\|\big({\rm w}^n_x+({\rm w}^n_0)_x\big)\delta_x\|_{B^{s-2}_{p,r}}\leq C\big(\|{\rm w}^n\|_{B^{s}_{p,r}}+\|{\rm w}^n_0\|_{B^{s}_{p,r}}\big)\|\delta\|_{B^{s-1}_{p,r}},\\
&\|{\rm{p}}_x\ast\big(\big({\rm w}^n+{\rm w}^n_0\big)\delta\big)\|_{B^{s-1}_{p,r}}\leq C\|{\rm w}^n+{\rm w}^n_0\|_{B^{s-2}_{p,r}}\|\delta\|_{B^{s-1}_{p,r}}\leq C\big(\|{\rm w}^n\|_{B^{s}_{p,r}}+\|{\rm w}^n_0\|_{B^{s}_{p,r}}\big)\|\delta\|_{B^{s-1}_{p,r}},\\
&\|{\rm{p}}_x\ast\big(\big({\rm w}^n\big)^2\delta\big)\|_{B^{s-1}_{p,r}}\leq C\|{\rm w}^n\|^2_{B^{s}_{p,r}}\|\delta\|_{B^{s-1}_{p,r}},\\
&\|{\rm{p}}_x\ast\big(\big({\rm w}^n+{\rm w}^n_0\big){\rm w}_0^n\delta\big)\|_{B^{s-1}_{p,r}}\leq C\|\big({\rm w}^n+{\rm w}^n_0\big){\rm w}_0^n\|_{B^{s-2}_{p,r}}\|\delta\|_{B^{s-1}_{p,r}}\leq C\big(\|{\rm w}^n\|^2_{B^{s}_{p,r}}+\|{\rm w}^n_0\|^2_{B^{s}_{p,r}}\big)\|\delta\|_{B^{s-1}_{p,r}},\\
&\|{\rm{p}}_x\ast\big(\big(({\rm w}^n)^2+({\rm w}^n_0)^2\big)\big({\rm w}^n+{\rm w}^n_0\big)\delta\big)\|_{B^{s-1}_{p,r}}\leq C\big(\|{\rm w}^n\|^2_{B^{s}_{p,r}}+\|{\rm w}^n_0\|^2_{B^{s}_{p,r}}\big)\big(\|{\rm w}^n\|_{B^{s}_{p,r}}+\|{\rm w}^n_0\|_{B^{s}_{p,r}}\big)\|\delta\|_{B^{s-1}_{p,r}}.
\end{align*}
It follows that 
\begin{align*}
\|{\rm w}^n-{\rm w}^n_0\|_{B^{s-1}_{p,r}}=\|\delta\|_{B^{s-1}_{p,r}}\leq C2^{n(\frac{1}{2}-s)},	
\end{align*}
which impliesfrom the interpolation inequality that
\begin{align*}
\|{\rm w}^n-{\rm w}^n_0\|_{B^{s}_{p,r}}\leq\|{\rm w}^n-{\rm w}^n_0\|^{\frac{1}{2}}_{B^{s-1}_{p,r}}\|{\rm w}^n-{\rm w}^n_0\|^{\frac{1}{2}}_{B^{s+1}_{p,r}}\leq C2^{\frac{1}{2}(\frac{1}{2}-s)n}2^{\frac{n}{2}}\leq C2^{-\frac{1}{2}(s-\frac{3}{2})n}.	
\end{align*}
This thus finish the proof of the proposition.

\end{proof}

In order to obtain the non-uniform continuous dependence for Eq. \eqref{scaling}, we need to construct a sequence of initial data $u_0^n={\rm w}^n_0+{\rm v}^n_0$.
\begin{prop}
Let $(s,p,r)$ meet the condition \eqref{cond} in Theorem \ref{uniform} and $u^n$ be the solution to \eqref{scaling} with initial data $u^n_0$. Define ${\rm z}_0^n=-u^n_0\partial_{x}u^n_0$. Then we have
\begin{align}
	\|u^n-u^n_0-t{\rm z}_0^n\|_{B^{s}_{p,r}}\leq Ct^2+C2^{-n\min\big\{s-\frac 3 2,1\big\}}.\label{wwn}
\end{align}
\end{prop}
\begin{proof}
Since $u_0^n={\rm w}^n_0+{\rm v}_0^n$, then by Lemma \ref{wsin}--\ref{vcos} and Proposition \ref{wnw0}, we have
\begin{align}\label{unn}
\|u^n\|_{B^{s+k}_{p,r}}\leq C\|u^n_0\|_{B^{s+k}_{p,r}}\leq C2^{kn},\quad k=0,\pm1.	
\end{align}
Owing to ${\rm z}_0^n=-u^n_0\partial_{x}u_0^n$, we can deduce that 
\begin{align*}
&\|{\rm z}^n_0\|_{B^{s-1}_{p,r}}\leq C\|u^n_0\|_{B^{s-1}_{p,r}}\|\partial_{x}u^n_0\|_{B^{s-1}_{p,r}}\leq C\|u^n_0\|_{B^{s-1}_{p,r}}\|u^n_0\|_{B^{s}_{p,r}}\leq C2^{-n},\\
&\|{\rm z}^n_0\|_{B^{s}_{p,r}}\leq C\|u^n_0\|_{L^\infty}\|u^n_0\|_{B^{s+1}_{p,r}}+\|\partial_{x}u^n_0\|_{L^\infty}\|u^n_0\|_{B^{s}_{p,r}}\leq C2^{-n}2^n+C\leq C,\\
&\|{\rm z}^n_0\|_{B^{s+1}_{p,r}}\leq C\|u^n_0\|_{L^\infty}\|u^n_0\|_{B^{s+2}_{p,r}}+\|\partial_{x}u^n_0\|_{L^\infty}\|u^n_0\|_{B^{s+1}_{p,r}}\leq C2^{-n}2^{2n}+C2^n\leq C2^{n}.	
\end{align*}
Set $u^n=u_0^n+t{\rm z}_0^n+{\rm{e}}^n$. We know that ${\rm{e}}^n$ satisfies
\begin{equation}\label{enn}
\left\{\begin{array}{ll}
\partial_{t}{\rm{e}}^n+u^n\partial_{x}{\rm{e}}^n
=&-{\rm{e}}^n\partial_{x}(u^n_0+t{\rm z}^n_0)+\mathcal{H}_1+\mathcal{H}_2-t\big(u^n_0\partial_{x}{\rm z}^n_0+{\rm z}^n_0\partial_{x}u^n_0-\mathcal{H}_3\big)\\
&-t^2\big({\rm z}^n_0\partial_{x}{\rm z}^n_0-\mathcal{H}_4\big)+\mathcal{H}_5,\\
{\rm{e}}^n(0,x)=0
\end{array}\right.
\end{equation}
where 
\begin{align*}
&\mathcal{H}_1=-\partial_{x}{\rm{p}}\ast\Big(c_1(u^n)^2+c_2(u^n)^3+c_3(u^n)^4\Big),\\
&\mathcal{H}_2=-\partial_{x}{\rm{p}}\ast\Big(\frac{1}{2}\big(\partial_{x}u^n_0\big)^2\Big),\\
&\mathcal{H}_3=-\partial_{x}{\rm{p}}\ast\big(\partial_{x}u^n_0\partial_{x}{\rm z}^n_0\big),\\
&\mathcal{H}_4-\partial_{x}{\rm{p}}\ast\Big(\frac{1}{2}\big(\partial_{x}{\rm z}^n_0\big)^2\Big),\\
&\mathcal{H}_5=-\partial_{x}{\rm{p}}\ast\Big(\partial_{x}{\rm{e}}^n\big(\partial_{x}u^n+\partial_{x}(u_0^n+t{\rm z}_0^n)\big)\Big).
\end{align*}
For $k=-1$,
\begin{align*}
\|{\rm{e}}^n\partial_{x}(u^n_0+t{\rm z}^n_0)\|_{B^{s-1}_{p,r}}\leq&\|{\rm{e}}^n\|_{B^{s-1}_{p,r}}\|u^n_0,{\rm z}^n_0\|_{B^{s-1}_{p,r}}\\
\leq&C\|{\rm{e}}^n\|_{B^{s+k}_{p,r}}\|u^n_0,{\rm z}^n_0\|_{B^{s}_{p,r}}+C(k+1)\|{\rm{e}}^n\|_{B^{s-1}_{p,r}}\|u^n_0,{\rm z}^n_0\|_{B^{s}_{p,r}},
\end{align*}
 and for $k=0$,
\begin{align*}
\|{\rm{e}}^n\partial_{x}(u^n_0+t{\rm z}^n_0)\|_{B^{s}_{p,r}}\leq&\Big(\|{\rm{e}}^n\|_{B^{s}_{p,r}}\|\partial_{x}u^n_0,\partial_{x}{\rm z}^n_0\|_{L^\infty}+\|\partial_{x}u^n_0,\partial_{x}{\rm z}^n_0\|_{B^{s}_{p,r}}\|{\rm{e}}^n\|_{L^\infty}\Big)\\
\leq&C\|{\rm{e}}^n\|_{B^{s}_{p,r}}\|u^n_0,{\rm z}^n_0\|_{B^{s}_{p,r}}+C\|{\rm{e}}^n\|_{B^{s-1}_{p,r}}\|u^n_0,{\rm z}^n_0\|_{B^{s+1}_{p,r}}\\
\leq&C\|{\rm{e}}^n\|_{B^{s+k}_{p,r}}\|u^n_0,{\rm z}^n_0\|_{B^{s}_{p,r}}+C(k+1)\|{\rm{e}}^n\|_{B^{s-1}_{p,r}}\|u^n_0,{\rm z}^n_0\|_{B^{s}_{p,r}}. 
\end{align*}
Applying Lemma \ref{priori estimate} to \eqref{enn}, we see for all $t\in[0,T]$,
\begin{align}
\|{\rm{e}}^n(t)\|_{B^{s+k}_{p,r}}
\leq&  C\int_0^t\|{\rm{e}}^n(\tau)\|_{B^{s+k}_{p,r}}\|u^n,u^n_0,{\rm z}^n_0\|_{B^{s}_{p,r}}{\ud}\tau+C(k+1)\int_0^t\|{\rm{e}}^n(\tau)\|_{B^{s-1}_{p,r}}\|u^n_0,{\rm z}^n_0\|_{B^{s}_{p,r}}{\ud}\tau\notag\\
&+Ct\|\mathcal{H}_1,\mathcal{H}_2\|_{B^{s+k}_{p,r}}+Ct^2\|u^n_0\partial_{x}{\rm z}^n_0,{\rm z}^n_0\partial_{x}u^n_0,\mathcal{H}_3\|_{B^{s+k}_{p,r}}+Ct^3\|{\rm z}^n_0\partial_{x}{\rm z}^n_0,\mathcal{H}_4\|_{B^{s+k}_{p,r}}.\label{ensk}
\end{align}
Noting that $u^n_0={\rm w}^n_0+{\rm v}^n_0$, we find 
\begin{align*}
\mathcal{H}_2=-\frac{1}{2}\partial_{x}{\rm{p}}\ast\Big(\big(\partial_{x}{\rm w}^n_0\big)^2\Big)-\frac{1}{2}\partial_{x}{\rm{p}}\ast\Big(\big(\partial_{x}{\rm v}^n_0\big)^2\Big)-\partial_{x}{\rm{p}}\ast\Big(\partial_{x}{\rm w}^n_0\partial_{x}{\rm v}^n_0\Big).
\end{align*}
Since $\partial_{x}{\rm{p}}$ is a $S^{-1}$ multiplier and $u^n$ is bounded in $C\Big([0,T];B^s_{p,r}\Big)$ with $T\approx1$, we can deduce from  Lemma \ref{product}, Lemma \ref{wsin} and Lemma \ref{vcos} that for $k=-1$,
\begin{align*}
&\|\partial_{x}{\rm{p}}\ast\Big(\big(\partial_{x}{\rm w}^n_0\big)^2\Big)\|_{B^{s-1}_{p,r}}\leq\|\partial_{x}p\ast\Big(\big(\partial_{x}{\rm w}^n_0\big)^2\Big)\|_{B^{s-\frac{1}{2}}_{p,r}}\leq\|\big(\partial_{x}{\rm w}^n_0\big)^2\|_{B^{s-\frac{3}{2}}_{p,r}}\leq C2^{\big(\frac{1}{2}-s\big)n},\\
&\|\partial_{x}{\rm{p}}\ast\Big(\big(\partial_{x}{\rm v}^n_0\big)^2\Big)\|_{B^{s-1}_{p,r}}\leq\|\partial_{x}{\rm v}^n_0\|^2_{B^{s-1}_{p,r}}\leq C2^{-2n},\\
&\|\partial_{x}{\rm{p}}\ast\Big(\partial_{x}{\rm w}^n_0\partial_{x}{\rm v}^n_0\Big)\|_{B^{s-1}_{p,r}}\leq C\|\partial_{x}{\rm w}^n_0\partial_{x}{\rm v}^n_0\|_{B^{s-2}_{p,r}}\leq C\|\partial_{x}{\rm w}^n_0\|_{B^{s-2}_{p,r}}\|\partial_{x}{\rm v}^n_0\|_{B^{s-1}_{p,r}}\\
&\qquad\qquad\qquad\qquad\qquad\quad\leq C2^{n(s-2)}2^{-ns+n}2^{-n}\leq C2^{-2n},\\
&\|\mathcal{H}_2\|_{B^{s-1}_{p,r}}\leq C2^{-n\min\{s-\frac{1}{2},2\}},\\
&\|\mathcal{H}_1\|_{B^{s-1}_{p,r}}\leq C\|u^n\|^l_{B^{s-1}_{p,r}}\leq C2^{-ln},\ l=2,3,4,\\
&\|\mathcal{H}_3\|_{B^{s-1}_{p,r}}+\|\mathcal{H}_4\|_{B^{s-1}_{p,r}}\leq C\Big(\|\partial_{x}u^n_0\|_{B^{s-2}_{p,r}}+\|\partial_{x}{\rm z}^n_0\|_{B^{s-2}_{p,r}}\Big)\|\partial_{x}{\rm z}^n_0\|_{B^{s-1}_{p,r}}\\
&\qquad\qquad\qquad\qquad\qquad\leq C\Big(\|u^n_0\|_{B^{s-1}_{p,r}}+\|{\rm z}^n_0\|_{B^{s-1}_{p,r}}\Big)\|{\rm z}^n_0\|_{B^{s}_{p,r}}\leq C2^{-n},\\
&\|u^n_0\partial_{x}{\rm z}^n_0+{\rm z}^n_0\partial_{x}u^n_0\|_{B^{s-1}_{p,r}}\leq C\Big(\|u^n_0\|_{B^{s-1}_{p,r}}\|{\rm z}^n_0\|_{B^{s}_{p,r}}+\|u^n_0\|_{B^{s}_{p,r}}\|{\rm z}^n_0\|_{B^{s-1}_{p,r}}\Big)\leq C2^{-n},\\
&\|{\rm z}^n_0\partial_{x}{\rm z}^n_0\|_{B^{s-1}_{p,r}}\leq C\|{\rm z}^n_0\|_{B^{s-1}_{p,r}}\|{\rm z}^n_0\|_{B^{s}_{p,r}}\leq C2^{-n}.
\end{align*}
Plugging the above inequaties into \eqref{ensk}, we can find 
\begin{align}
\|{\rm{e}}^n(t)\|_{B^{s-1}_{p,r}}\leq& \int_0^t(\tau 2^{-n}+\tau^22^{-n}+C2^{-n\min\{s-\frac{1}{2},2\}}){\ud}\tau+C\int_0^t\|{\rm{e}}^n(\tau)\|_{B^{s-1}_{p,r}}{\ud}\tau\notag\\
\leq&C2^{-n\min\{s-\frac{1}{2},2\}}+Ct^22^{-n}+C\int_0^t\|{\rm{e}}^n(\tau)\|_{B^{s-1}_{p,r}}{\ud}\tau\notag\\
\leq&C2^{-n\min\{s-\frac{1}{2},2\}}+Ct^22^{-n}\label{enk-1}
\end{align}
which we use the Gronwall lemma in the last inequality.

For $k=0$, similarly, by use of Lemma \ref{product}, Lemma \ref{wsin} and Lemma \ref{vcos}, one has 
\begin{align*}
&\|\partial_{x}{\rm{p}}\ast\Big(\big(\partial_{x}{\rm w}^n_0\big)^2\Big)\|_{B^{s}_{p,r}}\leq\|\big(\partial_{x}{\rm w}^n_0\big)^2\|_{B^{s-1}_{p,r}}\leq\|\partial_{x}{\rm w}^n_0\|_{B^{s-1}_{p,r}}\|\partial_{x}{\rm w}^n_0\|_{L^\infty}\\
&\qquad\qquad\qquad\qquad\qquad\leq C2^{n(s-1)}2^{-ns+n}2^{-ns+n}\leq C2^{(1-s)n},\\
&\|\partial_{x}{\rm{p}}\ast\Big(\big(\partial_{x}{\rm v}^n_0\big)^2\Big)\|_{B^{s}_{p,r}}\leq\|\partial_{x}{\rm v}^n_0\|^2_{B^{s-1}_{p,r}}\leq C2^{-2n},\\
&\|\partial_{x}{\rm{p}}\ast\Big(\partial_{x}{\rm w}^n_0\partial_{x}{\rm v}^n_0\Big)\|_{B^{s}_{p,r}}\leq C\|\partial_{x}{\rm w}^n_0\partial_{x}{\rm v}^n_0\|_{B^{s-1}_{p,r}}\\
&\qquad\qquad\qquad\qquad\qquad\quad\leq C\Big(\|\partial_{x}{\rm w}^n_0\|_{B^{s-1}_{p,r}}\|\partial_{x}{\rm v}^n_0\|_{L^\infty}+\|\partial_{x}{\rm v}^n_0\|_{B^{s-1}_{p,r}}\|\partial_{x}{\rm w}^n_0\|_{L^\infty}\Big)\leq C2^{-n},\\
&\|\mathcal{H}_2\|_{B^{s}_{p,r}}\leq C2^{-n\min\{s-1,1\}},\\
&\|\mathcal{H}_1\|_{B^{s}_{p,r}}\leq C\|u^n\|^l_{B^{s}_{p,r}}\leq C,\ l=2,3,4,\\
&\|\mathcal{H}_3\|_{B^{s}_{p,r}}+\|\mathcal{H}_4\|_{B^{s}_{p,r}}\leq C\Big(\|\partial_{x}u^n_0\|_{B^{s-1}_{p,r}}+\|\partial_{x}{\rm z}^n_0\|_{B^{s-1}_{p,r}}\Big)\|\partial_{x}{\rm v}^n_0\|_{B^{s-1}_{p,r}}\leq C,\\
&\|u^n_0\partial_{x}{\rm z}^n_0+{\rm z}^n_0\partial_{x}u^n_0\|_{B^{s}_{p,r}}\leq C\Big(\|u^n_0\|_{B^{s}_{p,r}}\|{\rm z}^n_0\|_{B^{s}_{p,r}}+\|u^n_0\|_{B^{s-1}_{p,r}}\|{\rm z}^n_0\|_{B^{s+1}_{p,r}}+\|u^n_0\|_{B^{s+1}_{p,r}}\|{\rm z}^n_0\|_{B^{s-1}_{p,r}}\Big)\leq C,\\
&\|{\rm z}^n_0\partial_{x}{\rm z}^n_0\|_{B^{s}_{p,r}}\leq C\Big(\|{\rm z}^n_0\|_{B^{s}_{p,r}}\|\partial_{x}{\rm z}^n_0\|_{L^\infty}+\|{\rm z}^n_0\|_{L^\infty}\|\partial_{x}{\rm z}^n_0\|_{B^{s}_{p,r}}\Big)\leq C\Big(\|{\rm z}^n_0\|^2_{B^{s}_{p,r}}+\|{\rm z}^n_0\|_{B^{s+1}_{p,r}}\|{\rm z}^n_0\|_{B^{s-1}_{p,r}}\Big)\leq C.
\end{align*}
Taking advantage of \eqref{ensk} and \eqref{enk-1}, we can easily obtain 
\begin{align*}
\|{\rm{e}}^n(t)\|_{B^{s}_{p,r}}\leq&C\int_0^t2^n\|{\rm{e}}^n(\tau)\|_{B^{s-1}_{p,r}}{\ud}\tau+Ct^2+Ct2^{-n\min\{s-1,1\}}+\int_0^t\|{\rm{e}}^n(\tau)\|_{B^{s}_{p,r}}{\ud}\tau\\
\leq&Ct2^{-n\min\{s-\frac{3}{2},1\}}+Ct^2+Ct2^{-n\min\{s-1,1\}}+\int_0^t\|{\rm{e}}^n(\tau)\|_{B^{s}_{p,r}}{\ud}\tau\\
\leq&Ct^2+C2^{-n\min\{s-\frac{3}{2},1\}}.
\end{align*}
Thus, we prove the proposition.  
\end{proof}
\begin{prop}\label{critial}
Let $(s,p,r)$ satisfy the consdition \eqref{cond1} in Theorem \ref{uniform1}. Assume that $\|u^{n}_{0}\|_{B^{1+\frac{1}{p}}_{p,1}}\lesssim1$ and $u$ is the solution to \eqref{scaling} with initial data $u^n_0$, then we obtain 
\begin{align}
\|u^{n}-u^{n}_{0}-t\mathbf{\rm h}(u^{n}_{0})\|_{B^{1+\frac{1}{p}}_{p,1}}\leq Ct^{2}\mathbf{Q}(u^{n}_{0}),\label{1pp1}
\end{align} 
where $\mathbf{\rm h}(u^{n}_{0}):=G(u^{n}_{0})-u_{0}^{n}\partial_{x}u_{0}^{n}$ and 
\begin{align*}
\mathbf{Q}(u^{n}_{0})=&1+\|u_{0}^{n}\|_{L^{\infty}}\|u_{0}^{n}\|_{B^{2+\frac{1}{p}}_{p,1}}+\|u_{0}^{n}\|_{L^{\infty}}^{2}\|u_{0}^{n}\|_{B^{3+\frac{1}{p}}_{p,1}}+\big(\|u_{0x}^{n}\|_{L^{\infty}}+\|u_{0}^{n}\|_{L^{\infty}}+\|u_{0}^{n}\|_{L^{\infty}}^{2}+\|u_{0}^{n}\|_{L^{\infty}}^{3}\big)^{2}\|u_{0}^{n}\|_{B^{2+\frac{1}{p}}_{p,1}}\\
&+\|u_{0}^{n}\|_{L^{\infty}}\big(\|u_{0x}^{n}\|_{L^{\infty}}+\|u_{0}^{n}\|_{L^{\infty}}+\|u_{0}^{n}\|_{L^{\infty}}^{2}+\|u_{0}^{n}\|_{L^{\infty}}^{3}\big)^{2}\|u_{0}^{n}\|_{B^{3+\frac{1}{p}}_{p,1}}
\end{align*}
\end{prop}
\begin{proof}
By virtue of Theorem \ref{th}, there exists a small time $T=T(\|u_{0}^{n}\|_{B^{1+\frac{1}{p}}_{p,1}})$ such that the solution $u^{n}\in C\big([0,T];B^{1+\frac{1}{p}}_{p,1}\big)$ and
\begin{align}
\|u^{n}\|_{L^{\infty}\big([0,T];B^{1+\frac{1}{p}}_{p,1}\big)}\leq C\|u_{0}^{n}\|_{B^{1+\frac{1}{p}}_{p,1}}\leq C\label{u1pp1}
\end{align}
which together with Lemma \ref{bspr} yields that for any $t\in[0,T]$ and $\sigma\geq1+\frac{1}{p}$
\begin{align}
\|u^{n}\|_{L^{\infty}\big([0,T];B^{\sigma}_{p,1}\big)}\leq& C\|u_{0}^{n}\|_{B^{\sigma}_{p,1}}e^{C\int_{0}^{t}\|u_{x}^{n}\|_{L^{\infty}}+\|u^{n}\|_{L^{\infty}}+\|u^{n}\|_{L^{\infty}}^{2}+\|u^{n}\|_{L^{\infty}}^{3}{\ud}\tau}\notag\\
\leq& C\|u_{0}^{n}\|_{B^{\sigma}_{p,1}}e^{C\int_{0}^{t}\big(1+\|u^{n}\|_{B^{1+\frac{1}{p}}_{p,1}}\big)^{3}{\ud}\tau}\leq C\|u_{0}^{n}\|_{B^{\sigma}_{p,1}}.\label{u11pp1}
\end{align} 
Since $u^{n}\in C\big([0,T];B^{1+\frac{1}{p}}_{p,1}\big)\hookrightarrow C\big([0,T];C^{0,1}\big)$, then we have
\begin{align}
u^{n}(t)-u_{0}^{n}=\int_{0}^{t}\partial_{\tau}u^{n}{\ud}\tau.\label{untau}	
\end{align}
Setting $W^{n}=u^{n}(t)-u^{n}_{0}$, so $W^{n}$ solves 
\begin{equation}\label{unun0}
\left\{\begin{array}{l}
\partial_{t}W^n+u^n\partial_{x}u^n=G(u^{n}),\\
W^n(0,x)=0.
\end{array}\right.
\end{equation} 
According to \eqref{scaling}, \eqref{untau} and Lemma \ref{bspr}, we see
\begin{align}
\|W^{n}\|_{L^{\infty}}\leq&\int_{0}^{t}\|\partial_{\tau}u^{n}\|_{L^{\infty}}{\ud}\tau\leq\int_{0}^{t}\|u^{n}\partial_{x}u^{n}\|_{L^{\infty}}+\|G(u^{n})\|_{L^{\infty}}{\ud}\tau\notag\\
\leq&C\int_{0}^{t}\big(\|\partial_{x}u^{n}\|_{L^{\infty}}+\|u^{n}\|_{L^{\infty}}+\|u^{n}\|_{L^{\infty}}^{2}+\|u^{n}\|_{L^{\infty}}^{3}\big)^{2}{\ud}\tau\notag\\
\leq&Ct\big(\|\partial_{x}u^{n}_{0}\|_{L^{\infty}}+\|u^{n}_{0}\|_{L^{\infty}}+\|u^{n}_{0}\|_{L^{\infty}}^{2}+\|u^{n}_{0}\|_{L^{\infty}}^{3}\big)^{2}\label{uninfty}
\end{align} 
which implies that
\begin{align}
\|u^{n}\|_{L^{\infty}}\leq\|u^{n}_{0}\|_{L^{\infty}}+Ct\big(\|\partial_{x}u^{n}_{0}\|_{L^{\infty}}+\|u^{n}_{0}\|_{L^{\infty}}+\|u^{n}_{0}\|_{L^{\infty}}^{2}+\|u^{n}_{0}\|_{L^{\infty}}^{3}\big)^{2}.\label{un123}	
\end{align}
From Lemma \ref{product} and \eqref{un123}, we deduce
\begin{align}
\|W^{n}\|_{B^{1+\frac{1}{p}}_{p,1}}\leq&\int_{0}^{t}\|\partial_{\tau}u^{n}\|_{B^{1+\frac{1}{p}}_{p,1}}{\ud}\tau\notag\\
\leq&\int_{0}^{t}\|u^{n}\partial_{x}u^{n}\|_{B^{1+\frac{1}{p}}_{p,1}}+\|G(u^{n})\|_{B^{1+\frac{1}{p}}_{p,1}}{\ud}\tau\notag\\
\leq&C\int_{0}^{t}\|u^{n}\|_{L^{\infty}}\|u^{n}\|_{B^{2+\frac{1}{p}}_{p,1}}{\ud}\tau+C\int_{0}^{t}\big(\|u^{n}\|_{B^{1+\frac{1}{p}}_{p,1}}^{2}+\|u^{n}\|_{B^{1+\frac{1}{p}}_{p,1}}^{3}+\|u^{n}\|_{B^{1+\frac{1}{p}}_{p,1}}^{4}\big){\ud}\tau\notag\\
\leq&Ct\Big[\|u^{n}_{0}\|_{L^{\infty}}+t\big(\|\partial_{x}u^{n}_{0}\|_{L^{\infty}}+\|u^{n}_{0}\|_{L^{\infty}}+\|u^{n}_{0}\|_{L^{\infty}}^{2}+\|u^{n}_{0}\|_{L^{\infty}}^{3}\big)^{2}\Big]\|u^{n}_{0}\|_{B^{2+\frac{1}{p}}_{p,1}}\notag\\
&+Ct\big(\|u^{n}_{0}\|_{B^{1+\frac{1}{p}}_{p,1}}^{2}+\|u^{n}_{0}\|_{B^{1+\frac{1}{p}}_{p,1}}^{3}+\|u^{n}_{0}\|_{B^{1+\frac{1}{p}}_{p,1}}^{4}\big).\label{un1p1}
\end{align}
Taking advantage of Lemma \ref{product} again, we find
\begin{align}
\|W^{n}\|_{B^{2+\frac{1}{p}}_{p,1}}\leq&\int_{0}^{t}\|\partial_{\tau}u^{n}\|_{B^{2+\frac{1}{p}}_{p,1}}{\ud}\tau\notag\\
\leq&\int_{0}^{t}\|u^{n}\partial_{x}u^{n}\|_{B^{2+\frac{1}{p}}_{p,1}}+\|G(u^{n})\|_{B^{2+\frac{1}{p}}_{p,1}}{\ud}\tau\notag\\
\leq&C\int_{0}^{t}\|u^{n}\|_{L^{\infty}}\|u^{n}\|_{B^{3+\frac{1}{p}}_{p,1}}{\ud}\tau+C\int_{0}^{t}\|u^{n}\|_{B^{2+\frac{1}{p}}_{p,1}}\big(\|u^{n}\|_{B^{1+\frac{1}{p}}_{p,1}}+\|u^{n}\|_{B^{1+\frac{1}{p}}_{p,1}}^{2}+\|u^{n}\|_{B^{1+\frac{1}{p}}_{p,1}}^{3}\big){\ud}\tau\notag\\
\leq&Ct\Big[\|u^{n}_{0}\|_{L^{\infty}}+t\big(\|\partial_{x}u^{n}_{0}\|_{L^{\infty}}+\|u^{n}_{0}\|_{L^{\infty}}+\|u^{n}_{0}\|_{L^{\infty}}^{2}+\|u^{n}_{0}\|_{L^{\infty}}^{3}\big)^{2}\Big]\|u^{n}_{0}\|_{B^{3+\frac{1}{p}}_{p,1}}\notag\\
&+Ct\big(\|u^{n}_{0}\|_{B^{1+\frac{1}{p}}_{p,1}}+\|u^{n}_{0}\|_{B^{1+\frac{1}{p}}_{p,1}}^{2}+\|u^{n}_{0}\|_{B^{1+\frac{1}{p}}_{p,1}}^{3}\big)\|u^{n}_{0}\|_{B^{2+\frac{1}{p}}_{p,1}}.\label{un2p1}
\end{align}
Noting that $u^{n}-u^{n}_{0}-t\mathbf{\rm h}(u^{n}_{0})=\int_{0}^{t}\partial_{\tau}u^{n}-\mathbf{\rm h}(u^{n}_{0}){\ud}\tau$, Lemma \ref{product}, \eqref{uninfty}, \eqref{un1p1} and \eqref{un2p1} together yield
\begin{align*}
\|u^{n}-u^{n}_{0}-t\mathbf{\rm h}(u^{n}_{0})\|_{B^{1+\frac{1}{p}}_{p,1}}\leq&\int_{0}^{t}\|\partial_{\tau}u^{n}-\mathbf{\rm h}(u^{n}_{0})\|_{B^{1+\frac{1}{p}}_{p,1}}{\ud}\tau\\
\leq&\int_{0}^{t}\|u^{n}\partial_{x}u^{n}-u_{0}^{n}\partial_{x}u^{n}_{0}\|_{B^{1+\frac{1}{p}}_{p,1}}{\ud}\tau+\int_{0}^{t}\|G(u^{n})-G(u_{0}^{n})\|_{B^{1+\frac{1}{p}}_{p,1}}{\ud}\tau\\
\leq&C\int_{0}^{t}\|u^{n}(\tau)-u_{0}^{n}\|_{B^{1+\frac{1}{p}}_{p,1}}{\ud}\tau+C\int_{0}^{t}\|u^{n}(\tau)-u_{0}^{n}\|_{L^{\infty}}\|u^{n}\|_{B^{2+\frac{1}{p}}_{p,1}}{\ud}\tau\\
&+C\int_{0}^{t}\|u^{n}(\tau)-u_{0}^{n}\|_{B^{2+\frac{1}{p}}_{p,1}}\|u^{n}_{0}\|_{L^{\infty}}{\ud}\tau\\
\leq&C\int_{0}^{t}\|u^{n}(\tau)-u_{0}^{n}\|_{B^{1+\frac{1}{p}}_{p,1}}{\ud}\tau+C\|u_{0}^{n}\|_{B^{2+\frac{1}{p}}_{p,1}}\int_{0}^{t}\|u^{n}(\tau)-u_{0}^{n}\|_{L^{\infty}}{\ud}\tau\\
&+C\|u^{n}_{0}\|_{L^{\infty}}\int_{0}^{t}\|u^{n}(\tau)-u_{0}^{n}\|_{B^{2+\frac{1}{p}}_{p,1}}{\ud}\tau\notag\\
\leq&Ct^{2}\mathbf{Q}(u^{n}_{0}).
\end{align*}
This completes the proof of Proposition \ref{critial}.
\end{proof}

Now we give the proof of Theorem \ref{uniform} and Theorem \ref{uniform1}. 
\begin{proof}[{\rm\textbf{The proof of Theorem \ref{uniform}}}:]
Thanks to Lemma \ref{vcos}, we see
\begin{align*}
\|u^n_0-{\rm w}^n_0\|_{B^{s}_{p,r}}=\|{\rm v}^n_0\|_{B^{s}_{p,r}}\leq C2^{-n}\Big(\leq Ct^2+C2^{-\frac{1}{2}n\min\{s-\frac{3}{2},1\}}\Big).
\end{align*}
It follows that 
\begin{align*}
\lim\limits_{n\rightarrow\infty}\Big\|u^n_0-{\rm w}^n_0\Big\|_{B^{s}_{p,r}}=0.
\end{align*}
Moreover, we have from \eqref{vn0s}, \eqref{w} and \eqref{wwn}
\begin{align*}
\|u^n-{\rm w}^n\|_{B^{s}_{p,r}}=&\|t{\rm z}^n_0+{\rm v}^n_0+{\rm w}^n_0-{\rm w}^n+{\rm{e}}^n\|_{B^{s}_{p,r}}\\
\geq& t\|{\rm z}^n_0\|_{B^{s}_{p,r}}-\|{\rm v}^n_0\|_{B^{s}_{p,r}}-\|{\rm w}^n_0-{\rm w}^n\|_{B^{s}_{p,r}}-\|{\rm{e}}^n\|_{B^{s}_{p,r}}\\
\geq&t\|{\rm z}^n_0\|_{B^{s}_{p,r}}-Ct^2-C2^{-\frac{1}{2}n\min\{s-\frac{3}{2},1\}}.	
\end{align*}
Owing to ${\rm z}_0^n=-{\rm v}_0^n\partial_{x}{\rm w}_0^n-{\rm w}_0^n\partial_{x}{\rm w}_0^n-{\rm w}_0^n\partial_{x}{\rm v}_0^n-{\rm v}_0^n\partial_{x}{\rm v}_0^n$, we find from Lemma \ref{wsin} -- Lemma \ref{vcos}
\begin{align*}
&\|{\rm w}_0^n\partial_{x}{\rm w}_0^n\|_{B^{s}_{p,r}}\leq \|{\rm w}_0^n\|_{L^\infty}\|{\rm w}_0^n\|_{B^{s+1}_{p,r}}+\|\partial_{x}{\rm w}_0^n\|_{L^\infty}\|{\rm w}_0^n\|_{B^{s}_{p,r}}\leq C2^{-n(s-1)},\\
&\|{\rm w}_0^n\partial_{x}{\rm v}_0^n\|_{B^{s}_{p,r}}\leq\|{\rm w}_0^n\|_{B^{s}_{p,r}}\|{\rm v}_0^n\|_{B^{s+1}_{p,r}}\leq C2^{-n},\\
&\|{\rm v}_0^n\partial_{x}{\rm v}_0^n\|_{B^{s}_{p,r}}\leq\|{\rm v}_0^n\|_{B^{s}_{p,r}}\|{\rm v}_0^n\|_{B^{s+1}_{p,r}}\leq C2^{-2n},
\end{align*}
which implies that by the embedded inequality $B^{s}_{p,r}\hookrightarrow B^{s}_{p,\infty}$
\begin{align*}
-\|{\rm w}_0^n\partial_{x}{\rm w}_0^n\|_{B^{s}_{p,\infty}}-\|{\rm w}_0^n\partial_{x}{\rm v}_0^n\|_{B^{s}_{p,\infty}}-\|{\rm v}_0^n\partial_{x}{\rm v}_0^n\|_{B^{s}_{p,\infty}}\geq& -\|{\rm w}_0^n\partial_{x}{\rm w}_0^n\|_{B^{s}_{p,r}}-\|{\rm w}_0^n\partial_{x}{\rm v}_0^n\|_{B^{s}_{p,r}}-\|{\rm v}_0^n\partial_{x}{\rm v}_0^n\|_{B^{s}_{p,r}}\\
\geq&-C2^{-n\min\{s-1,1\}}.
\end{align*}
Hence,
\begin{align*}
\|u^n-{\rm w}^n\|_{B^{s}_{p,r}}
\geq&t\|{\rm v}^n_0\partial_{x}{\rm w}_0^n\|_{B^{s}_{p,\infty}}-C2^{-\frac{1}{2}n\min\{s-\frac{3}{2},1\}}\quad\text{for $t>0$ small enough.}
\end{align*}
Combining with \eqref{low}, we finally obtain that
\begin{align*}
\liminf\limits_{n\rightarrow\infty}\Big\|u^n-{\rm w}^n\Big\|_{B^{s}_{p,r}}\gtrsim t,	
\quad\text{for $t>0$ small enough.}
\end{align*}
This proves Theorem \ref{uniform}.
\end{proof}
\begin{proof}[{\rm\textbf{The proof of Theorem \ref{uniform1}}}:]
Owing to Lemma \ref{wsin} and Lemma \ref{vcos}, we have for $\theta\geq1+\frac{1}{p}$
\begin{align*}
&\|u_{0}^{n}\|_{B^{\theta}_{p,1}}\leq C2^{(\theta-1-\frac{1}{p})n},\ \|{\rm w}_{0}^{n}\|_{B^{\theta}_{p,1}}\leq C2^{(\theta-1-\frac{1}{p})n},\\
&\|u_{0}^{n}\|_{L^{\infty}}\leq C2^{-n},\ \|{\rm w}_{0}^{n}\|_{L^{\infty}}\leq C2^{-(1+\frac{1}{p})n}\leq C2^{-n},\\
&\|\partial_{x}u_{0}^{n}\|_{L^{\infty}}\leq C2^{-\frac{1}{p}n}\leq C2^{-\frac{1}{2}n},\ \|\partial_{x}{\rm w}_{0}^{n}\|_{L^{\infty}}\leq C2^{-\frac{1}{p}n}\leq C2^{-\frac{1}{2}n},
\end{align*}
which follows that
\begin{align*}
\mathbf{Q}(u^{n}_{0})\leq C,\qquad \mathbf{Q}({\rm w}_{0}^{n})\leq C.
\end{align*}
Since 
\begin{align*}
&u^{n}=u^{n}-u_{0}^{n}-t\mathbf{\rm h}(u^{n}_{0})+{\rm w}_{0}^{n}+{\rm v}_{0}^{n}+t\big(G(u_{0}^{n})-u_{0}^{n}\partial_{x}u_{0}^{n}\big),\\
&{\rm w}^{n}={\rm w}^{n}-{\rm w}_{0}^{n}-t\mathbf{\rm h}({\rm w}^{n}_{0})+{\rm w}_{0}^{n}+t\big(G({\rm w}^{n}_{0})-{\rm w}^{n}_{0}\partial_{x}{\rm w}^{n}_{0}\big),\\&u_{0}^{n}\partial_{x}u_{0}^{n}={\rm w}^{n}_{0}\partial_{x}{\rm w}^{n}_{0}+{\rm v}^{n}_{0}\partial_{x}{\rm w}^{n}_{0}+u^{n}_{0}\partial_{x}{\rm v}^{n}_{0},
\end{align*}
according to Lemma \ref{wsin}--Lemma \ref{vcos}, we see 
\begin{align*}
\|u_{0}^{n}\partial_{x}{\rm v}^{n}_{0}\|_{B^{1+\frac{1}{p}}_{p,1}}\leq C\|u_{0}^{n}\|_{B^{1+\frac{1}{p}}_{p,1}}\|{\rm v}^{n}_{0}\|_{B^{2+\frac{1}{p}}_{p,1}}\leq C2^{-n},\\
\|G(u_{0}^{n})-G({\rm w}^{n}_{0})\|_{B^{1+\frac{1}{p}}_{p,1}}\leq C\|{\rm v}^{n}_{0}\|_{B^{1+\frac{1}{p}}_{p,1}}\leq C2^{-n}.
\end{align*}
Hence,
\begin{align*}
&\|u^{n}-{\rm w}^{n}\|_{B^{1+\frac{1}{p}}_{p,1}}\\
=&\|\big(u^{n}-u_{0}^{n}-t\mathbf{\rm h}(u^{n}_{0})\big)-\big({\rm w}^{n}-{\rm w}_{0}^{n}-t\mathbf{\rm h}({\rm w}^{n}_{0})\big)+{\rm v}_{0}^{n}-t\big(G(u_{0}^{n})-G({\rm w}^{n}_{0})+{\rm v}^{n}_{0}\partial_{x}{\rm w}^{n}_{0}+u^{n}_{0}\partial_{x}{\rm v}^{n}_{0}\big)\|_{B^{1+\frac{1}{p}}_{p,1}}\\
\geq&t\|{\rm v}^{n}_{0}\partial_{x}{\rm w}^{n}_{0}\|_{B^{1+\frac{1}{p}}_{p,1}}-t\|G(u_{0}^{n})-G({\rm w}^{n}_{0}\|_{B^{1+\frac{1}{p}}_{p,1}}-t\|u^{n}_{0}\partial_{x}{\rm v}^{n}_{0}\|_{B^{1+\frac{1}{p}}_{p,1}}-\|{\rm v}^{n}_{0}\|_{B^{1+\frac{1}{p}}_{p,1}}\\
&-\|u^{n}-u_{0}^{n}
-t\mathbf{\rm h}(u^{n}_{0})\|_{B^{1+\frac{1}{p}}_{p,1}}
-\|{\rm w}^{n}-{\rm w}_{0}^{n}-t\mathbf{\rm h}({\rm w}^{n}_{0})\|_{B^{1+\frac{1}{p}}_{p,1}}\\
\geq&t\|{\rm v}^{n}_{0}\partial_{x}{\rm w}^{n}_{0}\|_{B^{1+\frac{1}{p}}_{p,\infty}}-Ct2^{-n}-C2^{-n}-Ct^{2},
\end{align*}
which together with \eqref{low} gives that
\begin{align*}
\liminf\limits_{n\rightarrow\infty}\Big\|u^n-{\rm w}^n\Big\|_{B^{1+\frac{1}{p}}_{p,1}}\gtrsim t,	
\quad\text{for $t>0$ small enough,}
\end{align*}
and Theorem \ref{uniform1} is proved.

\end{proof}

\noindent\textbf{Acknowledgements.}
Y. Guo was supported by the Guangdong Basic and Applied Basic Research Foundation (No. 2020A1515111092) and Research Fund of Guangdong-Hong Kong-Macao Joint Laboratory for Intelligent Micro-Nano Optoelectronic Technology (No. 2020B1212030010). X. Tu was supported by National Natural Science Foundation of China (No. 11801076).


\vspace{8mm}
E-mail address: guoyy35@fosu.edu.cn

E-mail address: tuxi@mail2.sysu.edu.cn

\end{document}